\documentclass[10pt,leqno]{article}
\usepackage{arxiv}
\usepackage{arxiv}
\usepackage{amsmath,amssymb,amsfonts,amsthm}
\usepackage{mathtools}
\usepackage{graphicx,xcolor}
\usepackage[utf8]{inputenc} 
\usepackage[T1]{fontenc}    
\usepackage{hyperref}       
\usepackage{url}            
\usepackage{booktabs}       
\usepackage{nicefrac}       
\usepackage{microtype}      
\usepackage{graphicx}
\usepackage{natbib}
\usepackage[version=4]{mhchem}
\usepackage{makecell}
\usepackage{comment}
\usepackage{xfrac} 
\usepackage{wasysym}
\newfont{\notapolice}{cmss8}
\newfont{\rempolice}{cmss9}
\newfont{\tablepolice}{cmtt10}
\newfont{\captionpolice}{cmsl9}
\newtheoremstyle{note}
{3pt}
{3pt}
{\rempolice}
{}
{\itshape}
{.}
{.5em}
{\thmname{#1}\thmnumber{ #2}\thmnote{ #3}}

\newtheorem{thm}{Theorem}[section]

\newtheorem{prop}[thm]{Proposition}

\theoremstyle{definition}
\newtheorem{defn}{Definition}[section]
\theoremstyle{note}
\newtheorem{remark}{Remark}

\newcommand{\jBV}[1]{{  j_{{#1}}^{\textrm{\tiny BV}}  }}
\newcommand{\iBV}[1]{{  i_{{#1}}^{\textrm{\tiny BV}}  }}

\newcommand{\bsinh}{{\textrm{\rm bsinh}}}
\newcommand{\bcosh}{{\textrm{\rm bcosh}}}
\newcommand{\absinh}{{\textrm{\rm arg\,bsinh}}}

\newcommand{\tend}{{  t_{\text{end}}  }}
\newcommand{\ovc}{{{\overline c}}}

\newcommand{\ovj}{{{\overline \jmath}}}
\newcommand{\ovphi}{{{\overline \phi}}}
\newcommand{\oveta}{{{\overline \eta}}}
\newcommand{\ovjBV}[1]{{ {\overline \jmath}_{{#1}}^{\textrm{\tiny BV}}  }}

\newcommand{\hatc}{{{\widehat c}}}

\title{An efficient 0D-space conservative and positivity preserving  battery model}
\date{} 					
\author{
    \href{https://orcid.org/0000-0003-2295-5118}{\includegraphics[height=1.6ex]{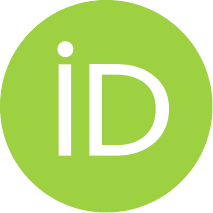}\hspace{1mm}S. Clain} \\
	Centre of Mathematics, Coimbra University, Largo D. Dinis, 3000-143 Coimbra, Portugal \\
	\texttt{clain@mat.uc.pt}\\
\And
	\href{https://orcid.org/0009-0008-6499-0097}{\includegraphics[scale=0.06]{orcid.pdf}\hspace{1mm}M.C. Lucas}\\
    Organic and Inorganic Chemistry Department, University of the Basque Country UPV/EHU,\\ P.O. Box 644, Bilbao, 48940, Spain\\
	\texttt{maria.mourato@ehu.eus}\\
\And
	\href{https://orcid.org/0000-0002-1900-8059}{\includegraphics[scale=0.06]{orcid.pdf}\hspace{1mm}M.T. Malheiro}\\
	Centre of Mathematics, Campus de Gualtar, 4710-057 Braga, Portugal \\
	\texttt{mtm@math.uminho.pt} \\
\And
	  \href{https://orcid.org/0000-}{\includegraphics[scale=0.06]{orcid.pdf}\hspace{1mm}C. M. Costa}\\
     Physics Center of Minho and Porto Universities (CF-UM-UP) and \\Laboratory of Physics for Materials and Emergent Technologies, LapMET,\\ University of Minho, Campus de Gualtar, 4710-057 Braga, Portugal \\
    \texttt{cmscosta@fisica.uminho.pt}\\
     }

\renewcommand{\undertitle}{Article}
\renewcommand{\shorttitle}{An efficient 0D-space conservative and positivity preserving  battery model}

\begin{document}
\maketitle
\begin{abstract}
Embedded software aims to provide the main characteristics of a cell in real time to manage and optimize battery usage. A complete multidimensional model is not tractable with the low computational resources available in-situ. A class of 0D-in-space models is derived from the P2D Doyle-Fuller-Newman system, with additional assumptions about the functions' shape in space, and strongly reduces the complexity while preserving the main characteristics over time. We propose a conservative, positivity-preserving model built from the P2D system based on the statement that the Butler-Volmer source term is constant in each subdomain (anode, separator, cathode). We develop a mathematically rigorous construction that provides the simplified model and design a scheme to solve the system. We highlight that the model guaranties the conservation property together with the positivity of the concentrations. In particular, the diffusion-limited current density that represents the maximum current that can be used while still preserving the positivity of the cation concentration is obtained explicitly.  
\end{abstract}
\section{Introduction}

The growing demand for electrochemical energy storage has shifted battery development from material-level optimization toward understanding and controlling complete electrochemical systems. Lithium-ion batteries, in particular, combine high energy and power density with relatively high efficiency and long cycle life, enabling their widespread use in portable electronics, electric vehicles, and stationary energy storage systems \cite{armand2008building, nitta2015liion}. However, the performance of a battery is not determined by a single material or electrochemical reaction. The cell results from the coupling of ionic transport in the electrolyte, electronic conductivity (in the solid phases), lithium diffusion inside the active material, charge-transfer reactions at the electrolyte--electrode interfaces, and the evolution of temperature and material properties during operation \cite{newman1962current, ramadesigan2012modeling}. These phenomena occur over different characteristic length and time scales and are strongly coupled through the local concentration, potential, current density, and reaction rate.

From a physical point of view, a rechargeable battery can therefore be regarded as a distributed electrochemical system rather than a simple electrical component. During charge and discharge, the imposed current generates spatial gradients of ionic concentration and electric potential in the electrolyte, while the electrochemical reaction transfers ions and electrons between the electrolyte and active materials. Inside the porous electrodes, the reaction is distributed over a large interfacial area and is coupled with diffusion inside the active particles. Consequently, the terminal voltage measured is only a macroscopic consequence of several internal processes. The description of these requires the simultaneous consideration of mass conservation, charge conservation, species transport, electrochemical kinetics, and thermodynamic equilibrium. 
This physical complexity explains the early need for mathematical models capable of connecting measurable quantities, such as current and voltage, with the internal states of the cell.

The historical development of battery modeling followed this need for quantitative description of internal cell behaviour. The theoretical analysis of porous electrodes established a framework in which current and reaction distribution could be described within a porous electrode rather than treating the electrode as a homogeneous surface \cite{newman1962current}. Doyle, Fuller and Newman introduced a physics-based model describing the coupled transport and electrochemical behaviour, followed by the formulation of the dual lithium ion insertion cell \cite{FDN94, DFN93}. These developments established what is now commonly referred to as the Doyle--Fuller--Newman (DFN), or pseudo-two-dimensional (P2D), framework. The importance of this approach is that the model not only reproduces the terminal voltage but also provides internal quantities. 

In the last three decades, modeling and numerical simulations have received increasing attention to develop, test, and provide in-depth information on all the complex mechanisms of the cell at any scale. Nowadays, many sophisticated battery models have been developed to address specific issues such as microstructure evolution, multiscale multi-domain coupling, electrochemical effects on the microscale, anisotropic behaviors, or thermal or mechanical degradation \cite{PAB22,ZSZ22,SSG23,APK26,MCS26}. All of these models require significant computational resources and generally do not aim to simulate or monitor batteries as a whole. 

On the other hand, macro-scale models, initiated with the Doyle-Fuller-Newman (DFN) model \cite{DFN93,FDN94,DN95,IMG23} describe the whole behavior of the cell by computing essential variables such as concentration, current density, and potential over a one-dimensional domain. It provides fine enough predictions regarding the major properties of the battery: simplified geometry, conductivity, diffusion, and concentration, among many others. Such modeling uses the one-dimensional variable $x$ following the normal vector direction of the cross-section in space together with the time variable $t$. Additionally, a pseudo-dimension $r$ is added to consider the effect of intercalation, assuming a spherical porous support of the electrodes. To reduce computational cost, we propose a mathematical reduction of the P2D equations to Single Particle Models (SPMe).

To sum up, any d-dimensional space model requires an advanced numerical solver based on Finite Element, Finite Difference, or Finite Volume methods equipped with, at least, a robust second-order time integrator. Simulations on modern computers take from a few minutes up to hours of running time and provide in-depth insight about the constituents of the cell. Nevertheless, the computational cost is still important for embedded computational resources, where the cells' evolution at the micro-nanoscale is not relevant. Modern battery packs feature embedded computing hardware that calculates, monitors, and predicts cell evolution in real time. Such models are fundamental to perform an efficient management of the cells that make up the whole battery. To deploy battery physics onto embedded computing hardware, we cannot use partial differential equation models, even for the less demanding Doyle-Fuller-Newman (DFN) model \cite{DV12,FL18}. They are too computationally demanding, as required by iterative solvers with numerous variables (several thousand). Moreover, fine-detail approximations are, in fact, out of the scope of a model that should handle the fundamental cells' information and their states. 

Simplification or Reduced Order Methods (ROM) aim at developing new mathematical and numerical models that reduce the effective time resolution while maintaining a good approximation (at least for the variable of interest). In particular, some critical properties such as conservation or concentration positivity have to be preserved. 

There exist several ways to reduce the complexity and computational cost using different techniques of ROM \cite{FPC15,LK22}. One of them \cite{SB07,SB09} consists of introducing a polynomial representation in space with time-dependent coefficients. The approximation is the summation of separated variables, {\it i.e.} functions, for instance, the product of time-dependent coefficients with polynomial functions in space (Galerkin method). The extreme situation corresponds to the consideration of a unique polynomial function in each subdomain (anode, cathode, and separator) but with different degrees corresponding to a simplified Galerkin method with only three elements. From another perspective, we employ a collocation technique on each subdomain using average and pointwise values at the interface of the physical quantities to provide a simpler model that enables handling of a few time-dependent unknowns.

Looking more cautiously at the literature, the choice of polynomials and their characterizations (degree, unknowns) is almost driven by intuition-based simplification \cite{DV12,DD18,YD22}. Moreover, most of the 0D models proposed in the literature do not necessarily fulfil fundamental physical properties such as conservation (charge, concentration) or positivity (\ce{Li^+} concentration). In particular, dealing with a 0D model that just uses average values as unknowns cannot avoid negative cation concentration in the dielectric at the collector's interface, leading to a non-physical model. 

In contrast, the simplified 0D model we propose follows a rigorous mathematical development to provide relations that involve average and interface values and respect the basic physical constraints. We apply an averaging procedure connecting mean values and interface values to preserve the conservation property (charge, moles, Kirchhoff law), concentrations' positivity, and provide an analytical expression of the limit diffusion intensity, {\it i.e.}, the maximum intensity such that the cations' concentration remains non-negative in the electrodes. Moreover, the model introduces the notion of target intensity (the intensity that the user would like to impose on the cell) that we compare with the real intensity (the intensity that the system can really deliver). 

To this end, we shall make two important statements to elaborate the model:
\begin{enumerate}
\item The variables follow a time-parametrized steady-state regime in space, {\it i.e.} the time derivatives are skipped, except for the concentrations in the electrodes (intercalated lithium in the anode and oxide-metallic lithium in the cathode).
\item We assume that the Butler-Volmer term is constant in space for each electrode, depending only on the concentrations' average.
\end{enumerate}
The first condition means that, except for some rare short transition situations, the time derivative of the current and the concentrations is neglected. The second condition seems quite restrictive, but the average value is a second-order approximation of a function over a small-sized domain (tens of micrometers). Therefore, we consider that, to deliver a simplified model with relevant accuracy, such a hypothesis is acceptable. One has to face a trade-off between simplicity and accuracy, and our assumption aims to balance the two antagonist objectives. As a result, the model is written as a time-dependent problem using average and pointwise values at the interface with around twenty variables. 

We prove that the constant (in space) Butler-Volmer source term assumption mathematically leads to a model where all the other variables (concentration, intensity, potential) are polynomials in space. Even better, we prove that the polynomial degree of the variables is not an intuition-based additional assumption, but a consequence of the second statement.

\section{The P2D model}
The standard P2D model \cite{DFN93,FDN94} consists of two porous electrodes separated by an intermediate layer, an ionically conducting separator, all filled with electrolyte. The negative electrode (anode) corresponds to a domain $\Omega_n=[x_{cn},x_{en}]$ and has length $L_n$, while the positive electrode (cathode) occupies $\Omega_p=[x_{ep},x_{cp}]$ and has length $L_p$. The intermediate layer (separator) occupies $\Omega_s=[x_{en},x_{ep}]$ and has length $L_s$. We set $x_{cn}=0$, so the total cell length is $L=L_n+L_s+L_p$ (see figure \ref{fig:P2D_figure}). 

\begin{figure}[ht]
    \centering
    \includegraphics[trim={0cm 5cm 0 0},clip,width=1\linewidth]{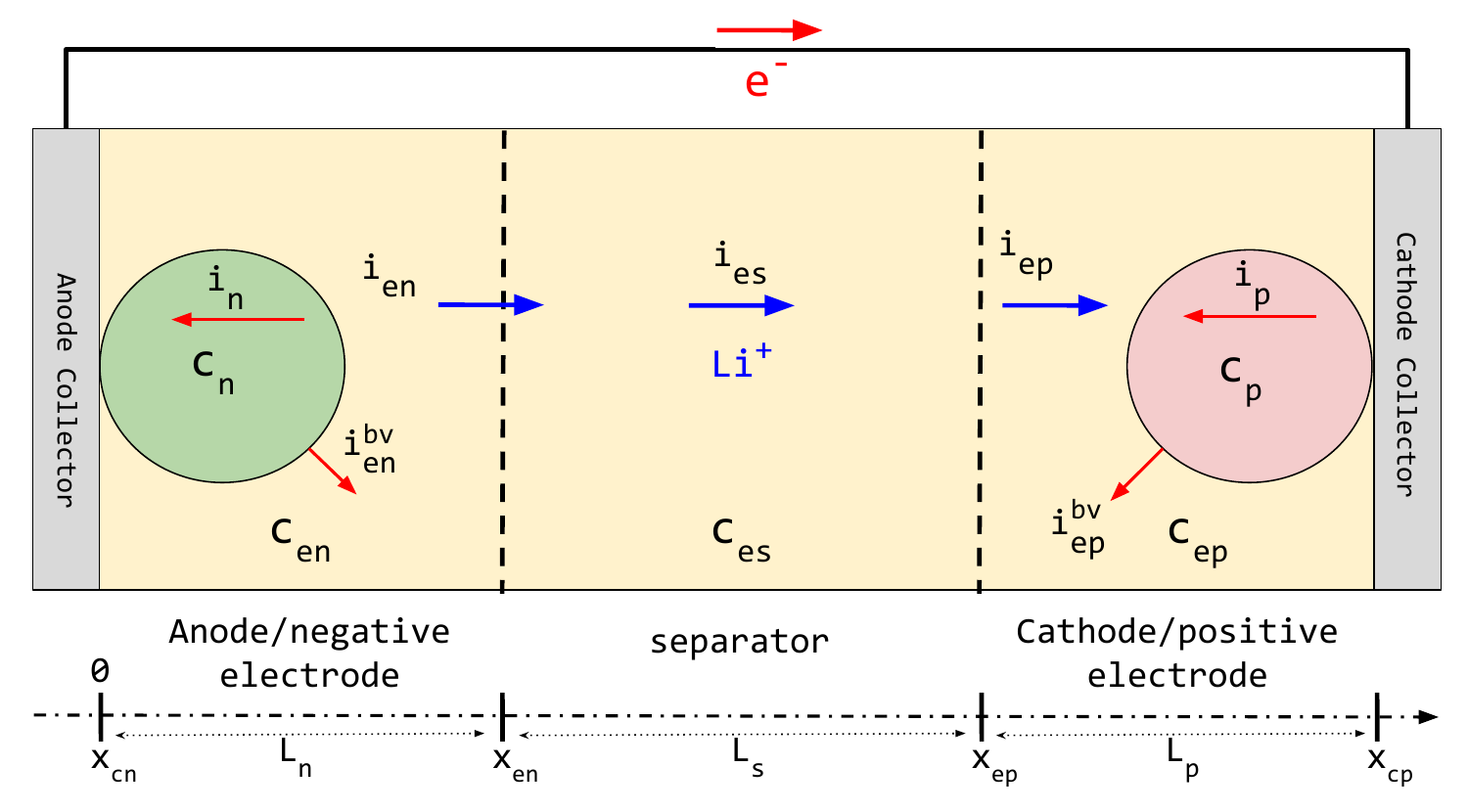}
    \caption{General design of a Lithium-ion cell during discharge}
    \label{fig:P2D_figure}
\end{figure}

In Table \ref{table::P2D_notations}, we provide the notation and variable names together with their dimension, used throughout the model. Note that $i$ $[A\,m^{-2}]$ denotes a surface current density, whereas $j$ $[A\,m^{-3}]$ represents a volume current density. 
\begin{table}[ht]
\begin{tabular}{| c |c| l |}
\hline
notation & dimension & name\\
\hline\hline
$c_n$, $c_p$& $[mol\,m^{-3}]$ & molar volume concentration in the negative and positive electrode  \\
$i_n$, $i_p$& $[A\,m^{-2}]$ & electronic current surface density in the negative and positive electrode  \\
$c_{en}$, $c_{ep}$, $c_{es}$ & $[mol\,m^{-3}]$ & molar concentration in the electrolyte\\
$\jBV{en}$, $\jBV{ep}$& $[A\,m^{-3}]$ & Butler-Volmer current volume density transfer in the negative and positive electrode \\
$i_{en}$, $i_{ep}$, $i_{es}$ & $[A\,m^{-2}]$ & cathodic current surface density of $\ce{Li^+}$ in the electrolyte\\
$\phi_n$, $\phi_p$& $[V]$ & electric potential in the negative and positive electrode  \\
$\Phi_{cn}$, $\Phi_{cp}$& $[V]$ & electric potential at the interface $x_{cn}$ and $x_{cp}$  \\
$\phi_{en}$, $\phi_{ep}$, $\phi_{es}$ & $[V]$ & ionic potential in the electrolyte \\
$\Phi_{en}$, $\Phi_{ep}$ & $[V]$ & ionic potential at the interfaces $x_{en}$ and $x_{ep}$ \\
$U_{en}$, $U_{ep}$ & $[V]$ & {open circuit potential of the negative and positive electrode}\\
$\eta_{ek}$, $\eta_{ek}$ & $[V]$ & overpotential  of the negative and positive electrode\\
$I$                      & $[A]$ & intensity between the positive and negative electrodes\\
\hline
$\alpha_n$, $\alpha_{p}$ & [-] &negative and positive electrode transfer coefficients\\
$a_n$, $a_p$ & $[m^{-1}]$ & active surface: ratio between the specific surface and the specific volume\\
$\varepsilon_n$, $\varepsilon_p$ , $\varepsilon_s$ & [-] & porosity of the negative, positive electrode and separator\\
$D_{n}$, $D_{p}$ & $[m^2\,s^{-1}]$ & molar diffusion in the electrodes\\
$D_{en}$, $D_{ep}$, $D_{es}$ & $[m^2\,s^{-1}]$ & molar diffusion of the electrolyte in the three subdomains\\
$A$& $[m^2]$ & orthogonal cut plane area of the apparatus\\
$\sigma_{en}$, $\sigma_{ep}$, $\sigma_{es}$&$[S\,m^{-1}]$ &ionic conductivity in each subdomain \\
\hline
\end{tabular}
\caption{Notations for the P2D cell model}
\label{table::P2D_notations}
\end{table}
\subsection{Species conservation}
The P2D model \cite{NT75,DFN93,FDN94,CDE22} couples macroscopic transport across the cell with microscopic diffusion inside active particles. The coordinate $x$ is the orthogonal direction to the cut-plane of the cell (macro-scale), $r$ is the radial coordinate describing the microscale diffusion effect of the oxidant/reduction component inside the porous electrodes, and $t$ denotes time.
Let $c_n(x,r,t)$ $[mol\,m^{-3}]$ be the molar concentration of the lithium atom in the negative electrode (usually \ce{C6Li}) and let $c_p(x,r,t)$ $[mol\,m^{-3}]$ denote the molar concentration of the lithium atom in the positive electrode (for example \ce{LiMnO2}, \ce{LiCoO2}). 

For each electrode $k\in\{n,p\}$ and each $x\in \Omega_k$, the conservation of species of the lithium compound is governed by Fick's diffusion law in the radial direction $r$:
\begin{eqnarray}\label{eq:single_ball_equation}
\partial_t c_{k} &=\displaystyle \frac{1}{r^2} \partial_r \biggl(r^2 D_{k} \partial_r c_{k}  \biggl),&\quad r\in \, ]0,R_k[.
\end{eqnarray}
Here, $R_k$ $[m]$ is the characteristic spherical radius of the particle and $D_{k}$ $[m^2\,s^{-1}]$ is the lithium diffusion coefficient in the electrode $k$.
The associated boundary conditions are
\begin{eqnarray}
D_{k}\partial_r c_{k}\Big|_{r=0}&=&0,\label{mole_r=0_BC}\\
D_{k}\partial_r c_{k} \Big|_{r=R_k} &=&-\frac{\iBV{ek}(x,t)}{F}= \displaystyle -\frac{\jBV{ek}(x,t)}{a_kF},\label{mole_r=R_BC}
\end{eqnarray}
where $a_k=3\varepsilon_k/R_k$ $[m^{-1}]$ is the  active surface, $\varepsilon_k$ $[m^3\,m^{-3}]$ is the electrode porosity, and $F$ $[C\, mol^{-1}]$ is Faraday's constant. 

The outward normal charge flux is controlled by the charge transfer per unit surface $\iBV{ek}(x,t)$ $[Am^{-2}]$, measured from the electrode toward the electrolyte. We recast it as a charge transfer per volume $\jBV{ek}(x,t)=a_k\,\iBV{ek}(x,t)$ $[A\,m^{-3}]$ given by the Butler-Volmer relation. With this convention, it is positive when \ce{Li -> Li+} (deintercalation), and negative when \ce{Li+ -> Li} (intercalation). 

\begin{remark}
The original Butler-Volmer (BV) relation \cite{DW20,RFR22,DV12} computes the exchange current density $\iBV{ek}$ per unit area $[Am^{-2}]$. Indeed, at the pore scale, the BV flux is given by a flux of charge across the interface between the electrode and the electrolyte. In a reduced one-dimensional model (the macroscopic average model), the volumetric current density $\jBV{ek}$ per unit volume $[Am^{-3}]$ is introduced as a source term in the species conservation equation by the macroscopic averaging (Volume of fluid (VOF) paradigm) used in a porous media system \cite{RBB12}. Consequently, in the literature, two current density versions coexist and may lead to some confusion in the applications. One has to carefully check whether the formulation uses a density current per unit area \cite{DW20} or per unit volume \cite{XNBD17}. \qedsymbol
\end{remark}
The conservation equation for \ce{Li+} in the electrolyte, with concentration  $c_{ek}(x,t)$ $[mol\,m^{-3}]$, $x \in \Omega_k$,  takes place in the three subdomains $\Omega_k$, $k\in\{n,p,s\}$ given by the relations
\begin{eqnarray}
\varepsilon_n \partial_t c_{en}&= \partial_x \biggl(D_{en} \partial_x c_{en} \biggl) +\displaystyle  \frac{(1-t^+)}{F} \jBV{en}&\quad x\in \Omega_n, \label{eq::conservation_Lin}\\
\varepsilon_p \partial_t c_{ep}&= \partial_x \biggl(D_{ep} \partial_x c_{ep} \biggl) +\displaystyle  \frac{(1-t^+)}{F} \jBV{ep}&\quad x\in \Omega_p, \label{eq::conservation_Lip}\\
\varepsilon_s \partial_t c_{es}&= \partial_x \biggl(D_{es} \partial_x c_{es} \biggl) &\quad x\in \Omega_s.\label{eq::conservation_Lis}
\end{eqnarray}
Here, $\varepsilon_s$ $[m^3\,m^{-3}]$ is the porosity of the separator, $t^+$ $[-]$ is the transfer number, and $D_{ek}$ $[m^2\,s^{-1}]$  is the effective diffusion of \ce{Li+} in the electrolyte filling the whole cell $\Omega$.

We prescribe that the cation concentration is continuous at the  electrode-separator interfaces $x=x_{en}$ and $x=x_{ep}$, while we state that no cation \ce{Li+} leaves the cell at the collector interface $x=x_{cn}$ and $x=x_{cp}$,
\begin{eqnarray}
&c_{en}(x_{en},t)=c_{es}(x_{en},t),\quad  &c_{ep}(x_{ep},t)=c_{es}(x_{ep},t),\label{eq_boundary_condition_Li_s}\\
&\partial_x c_{en}(x_{cn},t)=0, &\partial_x c_{ep}(x_{cp},t) =0.\label{eq_boundary_condition_Li_collector}
\end{eqnarray}

\subsection{Electronic current and potential}
For $k\in\{n,p\}$, we denote by $i_k=i_k(x,t)$, $[A\,m^{-2}]$ the electronic current density per unit area and $\phi_k=\phi_k(x,t)$,~$[V]$ the potential on the porous electrode $\Omega_k$. The potential is described by Ohm's law, which relates it to the current, while it depends on the electron transfer controlled by the Butler-Volmer source term.
\begin{eqnarray}
i_k&=&-\sigma_{k} \partial_x \phi_{k},\label{eq::potential_in electrode}\\ 
\partial_x i_k &=&-\jBV{ek}, \label{eq::electronic_charge_conservation}
\end{eqnarray}
with $\sigma_{k}$ $[S\,m^{-1}]$ the electronic conductivity. Moreover, we assume that no electron moves toward the electrolyte; hence, we prescribe
\begin{equation}\label{eq::electronic_flux_interface}
i_n(x_{en},t)=0,\qquad i_p(x_{ep},t)=0.
\end{equation}
Integrating \eqref{eq::electronic_charge_conservation} and using \eqref{eq::electronic_flux_interface}, we deduce the electronic flux at the collectors, 
\begin{eqnarray}
i_n(x_{cn},t)&=&+\int_{x_{cn}}^{x_{en}}\jBV{en}(x,t)\,dx, \label{eq::cutent_density_cn}\\
i_p(x_{cp},t)&=&-\int_{x_{ep}}^{x_{cp}}\jBV{ep}(x,t)\,dx. \label{eq::cutent_density_cp}
\end{eqnarray}
We define the potential at the collectors by
\begin{equation}\label{potencial_interface_collector}
\Phi_{cn}(t)=\phi_n(x_{cn},t),\qquad  \Phi_{cp}(t)=\phi_p(x_{cp},t).
\end{equation}
By convention, the electric current $I$ $[A]$ is the signed real number, positive when the electrons flow from the anode to the cathode, that is, when $\Phi_{cn}\leq \Phi_{cp}$. In this case, the electrons move out of $\Omega_n$ and we have $I(t)=A\,i_n(x_{cn},t)$, where $A\,[m^2]$ is the orthogonal cut plane area of the cell. 
On the other hand, the flux $i_p(x_{cp},t)$ represents the electronic density moving inward to $\Omega_p$. Charge conservation therefore gives $I(t)=A\,i_p(x_{cp},t)$. In conclusion, the current is given by
\begin{equation}\label{eq::current}
I(t)=A\,i_n(x_{cn},t)=A\,i_p(x_{cp},t).
\end{equation}
\begin{remark} For a discharging cell, $\jBV{en}>0$ and $\jBV{ep}<0$ means that we have the creation of electrons at the negative electrode, and the electrons are flowing from the anode towards the cathode with $I>0$. 
On the contrary, charging the cell means that $\jBV{en}<0$ and $\jBV{ep}>0$, we have a consumption of electrons at the negative electrode and $I<0$.\qedsymbol
\end{remark}
\subsection{Ionic current and potential}
The cation current density per unit area  $i_{ek}\equiv i_{ek}(x,t)$ $[Am^{-2}]$ contains a density current by conduction and a contribution driven by the \ce{Li+} concentration gradient. In the three subdomains, it is given by
\begin{eqnarray}
i_{en}&=&-\sigma_{en}\Big [\partial_x \phi_{en}+(2t^+-1)\frac{RT}{F}\partial_x \big [\ln(c_{en})\big ] \Big ], \quad x\in \Omega_n,\label{ionic_conservation_n}\\
i_{ep}&=&-\sigma_{ep}\Big [\partial_x \phi_{ep}+(2t^+-1)\frac{RT}{F}\partial_x \big [\ln(c_{ep})\big ]\Big ], \quad x\in \Omega_p,\label{ionic_conservation_p}\\
i_{es}&=&-\sigma_{es}\Big [\partial_x \phi_{es}+(2t^+-1)\frac{RT}{F}\partial_x \big [\ln(c_{es})\big ]\Big ], \quad x\in \Omega_s,\label{ionic_conservation_s}
\end{eqnarray}
with $\phi_{ek}(x,t)$ $[V]$, the ionic potential at $\Omega_k$,  $\sigma_{ek}$ $[Sm^{-1}]$ the ionic conductivity in $\Omega_k$,  $R$ $[Jmol^{-1}K^{-1}]$ is the universal gas constant and $T$ $[K]$ is the temperature. We assume the continuity of flux and potential at the electrode-separator interfaces
\begin{eqnarray}
&\phi_{en}(x_{en},t)=\phi_{es}(x_{en},t), &\phi_{ep}(x_{ep},t)=\phi_{es}(x_{ep},t),\\
&i_{en}(x_{en},t)=i_{es}(x_{en},t), &i_{ep}(x_{ep},t)=i_{es}(x_{ep},t). \label{ionic_flux_continuity_separator}
\end{eqnarray}
On the other hand, we assume that no cation leaves the cell at the collector's interface, so 
\begin{equation}\label{no_ion_transfer_at_collector}
i_{en}(x_{cn},t)=0,\qquad i_{ep}(x_{cp},t)=0.
\end{equation}
Charge conservation for the cation current density then reads
\begin{equation}\label{cation_density_current_conservation}
\partial_x i_{en}=\jBV{en},\quad \partial_x i_{es}=0,\quad \partial_x i_{ep}=\jBV{ep}.
\end{equation}
\subsection{Butler-Volmer equation}\label{sec::BVsusbsec}
The Butler-Volmer equation is the critical relation that describes charge transfer at the electrode-electrolyte interface. The  equation has an expression with separated variables as $\jBV{ek}=j^0_{ek}\,\cdot \chi_{ek}$, where $j^0_{ek}$ $[A\,m^{-3}]$, $k \in \{n,p\}$ is the contribution associated with the concentration of reactive compounds, while $\chi_{ek}$ $[-]$ depends only on the potentials. The usual expression of the reaction equation is \cite{DFN93,DW20}
$$
i^0_{ek}=F\,r_k\, \big (c_k\big )^{\alpha_k} \big (c_{ek} \big )^{1-\alpha_k} \big (c^{max}_k-c_k\big )^{1-\alpha_k},\qquad j_{ek}^0=a_k\,i^0_{ek},
$$
with $r_k$ the reaction rate of dimension $[mol^{(\alpha_k-1)} m^{(4-3\alpha_k)}s^{-1}]$, $i^0_{ek}$ $[A\,m^{-2}]$ the current density per surface unit and $j^0_{ek}$ the corresponding volumetric current density. The parameters $\alpha_{n},\ \alpha_{p}\in [0,1]$ $[-]$ are the anodic and cathodic transfer coefficients, respectively.

An alternative approach is based on a reference state (marked by the symbol $\star$) with concentrations $c_k^{\star}$, $c_{ek}^{\star}$. The corresponding  reference current densities are \cite{DW20}
$$
i^{0\star}_{ek}=F\,r_k\, \big (c_k^\star\big )^{\alpha_k} \big (c_{ek}^\star \big )^{1-\alpha_k} \big (c^{max}_k-c_k^\star\big )^{1-\alpha_k},\qquad 
j^{0\star}_{ek}=a_k\,i^{0\star}_{ek}
$$
 The exchange current density $j^0_{ek}$ is then written as a variation of this reference state
\begin{equation}\label{eq::exchange_density_current}
j^0_{ek}=j^{0\star}_{ek} 
\left ( \frac{c_k}{c_k^{\star}}\right )^{\alpha_k}
\left ( \frac{c_k^{max}-c_k}{c_k^{max}-c_k^{\star}}\right )^{1-\alpha_{k}}
\left ( \frac{c_{ek}}{c_{ek}^{\star}}\right )^{1-\alpha_{k}}.  
\end{equation}

On the other hand, the density current transfer is also controlled by the overpotential $\eta_{ek}$ $[V]$, given by
\begin{equation}\label{defn_overpotential}
\eta_{ek}=\phi_k-\phi_{ek}-U_{ek},
\end{equation}
where $U_{ek}$ $[V]$ is the open-circuit potential associated with the intercalation/deintercalation process. 

To express the function $\chi_{ek}$ compactly, we introduce the b-hyperbolic function $\bsinh$ (see section \ref{sec::b-trigonometry}) and set $f=\frac{F}{RT}$
\begin{align}\label{defn_chi}
\chi_{ek}(\eta_{ek})&=\exp\Big (\alpha_k f \eta_{ek}\Big )-\exp\Big ((\alpha_k-1) f \eta_{ek}\Big )\\
&=2\frac{\exp\Big (2\alpha_k f \eta_{ek}/2\Big )-\exp\Big (2(\alpha_k-1) f \eta_{ek}/2\Big )}{2} \nonumber\\
&=2\,\bsinh( f \eta_{ek}/2;\alpha_k).\nonumber
\end{align}
Therefore, the volumetric Butler-Volmer current density becomes
$$
\jBV{ek}=2\,j^0_{ek}\,\bsinh\big (f \eta_{ek}/2;\alpha_k\big ).
$$
\begin{remark}
When $\alpha_k=1/2$, we recover the hyperbolic sine function and the current density reads
$$
\jBV{ek}=2\,j^0_{ek}\,\sinh\big (f \eta_{ek}/2\big ). 
$$
For $\alpha\neq 1/2$ we define a new class of special functions $\bsinh(\eta,\beta)$ and $\bcosh(\eta,\beta)$, referred to here as $b$-hyperbolic trigonometry function, presented and studied in section \ref{sec::b-trigonometry}. \qedsymbol
\end{remark}

\section{The average procedure}
The P2D model 
depends on the spatial variable $x$ and the time $t$ together with an additional variable $r$ to provide a description of the electrode state at a microscopic level. Our objective is to reduce this spatially distributed description to a low-dimensional model. We therefore average the functions regarding $r$ and $x$, while retaining  the function values at specific points (collector--electrode interfaces $x_{cn}$, $x_{cp}$ and electrode-electrolyte $x_{en}$, $x_{ep}$).  

\subsection{Electrode species conservation}
We detail the average procedure for the negative electrode, but the same technique applies to the positive one. Integrating equation  \eqref{eq:single_ball_equation} over $[0,R_n]$ and using the boundary conditions \eqref{mole_r=0_BC}-\eqref{mole_r=R_BC}, we get
$$
\partial_t \int_0^{R_n} c_{n}(x,r,t)r^2 dr=  D_{n} (R_n)^2 \partial_r c_{n}(x,R_n,t)=-\frac{\jBV{en}(x,t)}{a_nF}(R_n)^2.
$$
Furthermore, let $c_n(x,t)$ denote the $r$-average of the function $c_n(x,r,t)$. We then have
\begin{eqnarray*}
\partial_t c_{n}(x,t)&=&\frac{3}{4\pi(R_n)^3}\times  \,4\pi\, \partial_t \int_0^{R_n} c_{n}(x,r,t)r^2 dr\\
&=&-\frac{3}{4\pi(R_n)^3} \times 4\pi\, \frac{\jBV{en}(x,t)}{a_n\,F}(R_n)^2,\\
&=&-\frac{3}{R_n}\frac{\jBV{en}(x,t)}{a_n\,F}.
\end{eqnarray*} 
Following \cite{XNBD17}, we take $a_k R_k=3\varepsilon_k$ and the $r$-average equation then reads
$$
\varepsilon_n \partial_t c_{n}(x,t)=-\frac{\jBV{en}(x,t)}{F}.
$$
The $x$-average $\ovc_n(t)$ comes from the relation over $\Omega_n$ and dividing by $L_n=|\Omega_n|$, we obtain the evolution equation for the spatially averaged concentration
$$
\frac{d\,\ovc_{n}}{dt} (t)=\frac{1}{L_n}\frac{d}{dt}\int_{x_{cn}}^{x_{en}}  c_{n}(x,t) dx=\frac{1}{L_n}\int_{x_{cn}}^{x_{en}}  \partial_t c_{n}(x,t) dx
=-\frac{1}{\varepsilon_nF\,L_n}\int_{x_{cn}}^{x_{en}} \jBV{en}(x,t)dx=-\frac{1}{\varepsilon_nF}\ovjBV{en}(t),
$$
where $\ovjBV{en}$ in $[A\,m^{-3}]$ is the average current transfer in the negative electrode and $\ovc_{n}$ ($[mol\,m^{-3}]$) is the corresponding average concentration of lithium. 
In summary, the variation of the average Lithium concentration in the electrode is given by
\begin{equation}\label{eq::anode_species_conc}
\varepsilon_n\frac{d\,\ovc_{n}}{dt} (t)=-\frac{\ovjBV{en}(t)}{F}.    
\end{equation}
Applying the same argument to the positive electrode, we deduce the corresponding species conservation equation in $\Omega_p$, 
\begin{equation}\label{eq::catode_species_conc}
\varepsilon_p\frac{d\,\ovc_{p}}{dt}(t)=-\frac{\ovjBV{ep}(t)}{F},    
\end{equation}
where $\ovc_{p}$ and $\ovjBV{ep}$ are the averages of $c_p$ and $\jBV{ep}$, respectively, in $\Omega_p$.
\begin{remark}
The well-balanced relation between the creation and consumption of \ce{Li} is then given by
\begin{equation}\label{eq_conservation_BV}
L_n\ \ovjBV{en}(t)+L_p\ \ovjBV{ep}(t)=0. \hskip3em \qedsymbol
\end{equation}
\end{remark}

\subsection{Electrolyte species conservation}
To obtain the $x$-average of the function $c_{en}(x,t)$, we consider the  \ce{Li+} ion conservation law in $\Omega_n$, equation \eqref{eq::conservation_Lin}. Integrating it over the electrode and using  the boundary condition \eqref{eq_boundary_condition_Li_collector} gives
\begin{eqnarray*}
\varepsilon_n L_n\frac{d\,\ovc_{en}}{dt} &=&\varepsilon_n \frac{d}{dt}  \int_{x_{cn}}^{x_{en}}c_{en}(x,t)\,dx
                                       =\int_{x_{cn}}^{x_{en}} \varepsilon_n \partial_t c_{en}(x,t)\,dx,\\
&=& \int_{x_{cn}}^{x_{en}}\partial_x \biggl(D_{en} \partial_x c_{en} \biggl) dx+
\displaystyle  \int_{x_{cn}}^{x_{en}} \frac{(1-t^+)}{F} \jBV{en}dx,\\
&=& D_{en}(\partial_x c_{en}(x_{en},t)- \partial_x c_{en}(x_{cn},t)) +\displaystyle  L_n\frac{(1-t^+)}{F} \ovjBV{en}(t),\\
&=& D_{en}\partial_x c_{en}(x_{en},t)+\displaystyle  \frac{(1-t^+)}{F} L_n\ovjBV{en}(t).
\end{eqnarray*}

Similarly, using the relations \eqref{eq::conservation_Lip} and \eqref{eq_boundary_condition_Li_collector}, we obtain the relation
\begin{eqnarray*}
\varepsilon_p L_p \frac{d\,\ovc_{ep}}{dt}(t) &=&-D_{ep}\partial_x c_{ep}(x_{ep},t)+\displaystyle  \frac{(1-t^+)}{F} L_p \ovjBV{ep}(t).
\end{eqnarray*}
\vskip 1em

Since there is no source term in the separator, integrating \eqref{eq::conservation_Lis} over $\Omega_s$ gives
$$
 L_s\varepsilon_s \frac{d}{dt} \ovc_{es}=D_{es}\partial_x c_{es}(x_{ep},t)-D_{es} \partial_x c_{es}(x_{en},t).
$$
At this stage, since no interfacial accumulation occurs, we assume the continuity of the cation flux at the interface $x=x_{en}$ and $x=x_{ep}$ given by condition \eqref{eq_boundary_condition_Li_s}, and we get the following
\begin{eqnarray*}
D_{es} \partial_x c_{es}(x_{en},t)&=&D_{en}\partial_x c_{en}(x_{en},t),\\
D_{es} \partial_x c_{es}(x_{ep},t)&=&D_{ep}\partial_x c_{ep}(x_{ep},t). 
\end{eqnarray*}
Adding the three relations and using property \eqref{eq_conservation_BV} provides  the global conservation of cations \ce{Li+} 
$$
\frac{d}{dt} \Big (L_n\varepsilon_n \ovc_{en}+L_s\varepsilon_s \ovc_{es}+L_p\varepsilon_p \ovc_{ep}\Big )=0.
$$
The well-balanced relation between the creation and consumption of cations is then given by
\begin{equation}\label{eq_conservation_Li+}
L_n\varepsilon_n \ovc_{en}(t)+L_s\varepsilon_s \ovc_{es}(t)+L_p\varepsilon_p \ovc_{ep}(t)=M^0,
\end{equation}
with $M^0$ calculated from the initial condition.

\subsection{Electronic current and potential}
\subsubsection{Electronic current in the electrodes}
Let $x\in \Omega_n$. From equation \eqref{eq::electronic_charge_conservation} and the boundary condition \eqref{eq::electronic_flux_interface}, integration over $[x,x_{en}]$ gives
\begin{equation}\label{integral_relation_j_n}
i_n(x,t)=i_n(x,t)-i_n(x_{en},t)=-\int_{x}^{x_{en}}\partial_x i_n(s,t)\,ds =+\int_{x}^{x_{en}} \jBV{en}(s,t)\,ds.
\end{equation}
Similarly, for any $x\in\Omega_p$ we integrate equation \eqref{eq::electronic_charge_conservation} over $[x_{ep},x]$ and  the boundary condition \eqref{eq::electronic_flux_interface} yields
\begin{equation}\label{integral_relation_j_p}
i_p(x,t)=i_p(x,t)-i_p(x_{ep},t)=\int_{x_{ep}}^x\partial_x i_p(s,t)\,ds=- \int_{x_{ep}}^{x} \jBV{ep}(s,t)\,ds.
\end{equation}

In particular,  from the relations \eqref{eq::cutent_density_cn}-\eqref{eq::cutent_density_cp} we conclude with \eqref{eq::current}
\begin{eqnarray}
I(t)&=&+A\,L_n \ \ovjBV{en}(t),\label{eq::average_intensity_n}\\ 
I(t)&=&-A\,L_p \, \ovjBV{ep}(t). \label{eq::average_intensity_p}
\end{eqnarray}

\subsubsection{Electric potential in the electrodes}
The potential is given by equation \eqref{eq::potential_in electrode} and we deduce from \eqref{integral_relation_j_n}
$$
\sigma_n\partial_x \phi_{n}(x,t)=-i_n(x,t)=-\int_{x}^{x_{en}} \jBV{en}(s,t)\,ds=-\int_{x_{cn}}^{x_{en}} \lambda^u(s,x)\jBV{en}(s,t)\,ds,
$$
where $\lambda^u(s,x)=1$ if $s>x$ and $\lambda^u(s,x)=0$ otherwise. 

Integrating the last relation over $\Omega_n=[x_{cn},x_{en}]$ provides the potentials at the extremity of the domain,
$$
\phi_n(x_{en},t)-\phi_n(x_{cn},t)=-\frac{1}{\sigma_n} 
\int_{x_{cn}}^{x_{en}} \int_{x_{cn}}^{x_{en}} \lambda^u(s,x) \jBV{en}(s,t)\,ds\,dx.
$$
From definition \eqref{potencial_interface_collector}, we deduce the electronic potential at the anode collector,  
$$
\Phi_{cn}(t)=\phi_n(x_{en},t)+\frac{1}{\sigma_n} \int_{x_{cn}}^{x_{en}} \int_{x_{cn}}^{x_{en}} \lambda^u(s,x)\jBV{en}(s,t)\,ds\,dx.
$$
Noting that $\displaystyle\int_{x_{cn}}^{x_{en}} \lambda^u(s,x)\,dx=s-x_{cn}$, we deduce
\begin{equation}
 \label{eq::elec_potential_anode}
\Phi_{cn}(t)=\phi_n(x_{en},t)+\frac{1}{\sigma_n} \int_{x_{cn}}^{x_{en}} (s-x_{cn})\jBV{en}(s,t)\,ds.
\end{equation}

Following the same technique and using the relation \eqref{eq::potential_in electrode} for the cathode, we deduce
$$
\sigma_p\partial_x \phi_{p}(x,t)=-i_{p}(x,t)=\int_{x_{ep}}^{x} \jBV{ep}(s,t)\,ds=\int_{x_{ep}}^{x_{cp}} \lambda^{\ell}(s,x)\jBV{ep}(s,t)\,ds
$$
with $\lambda^{\ell}(s,x)=1$ for $s<x$ and zero otherwise. From definition \eqref{potencial_interface_collector}, integration over $\Omega_p$ provides the relation
$$
\Phi_{cp}(t)=\phi_p(x_{ep},t)+\frac{1}{\sigma_p} \int_{x_{ep}}^{x_{cp}} \int_{x_{ep}}^{x_{cp}} \lambda^{\ell}(s,x)\jBV{ep}(s,t)\,ds\,dx.
$$
Noting that  $\displaystyle \int_{x_{ep}}^{x_{cp}} \lambda^{\ell}(s,x)\, dx=x_{cp}-s$, we get
\begin{equation}\label{eq::elec_potential_catode} 
\Phi_{cp}(t)=\phi_p(x_{ep},t)+\frac{1}{\sigma_p} \int_{x_{ep}}^{x_{cp}} (x_{cp}-s)\jBV{ep}(s,t)\,ds.
\end{equation}

\subsection{Ionic current and potential}

From the boundary conditions  $i_{en}(x_{cn},t)=0$, (see \eqref{no_ion_transfer_at_collector}) and  $\partial_x c_{en}(x_{cn},t)=0$, (see \eqref{eq_boundary_condition_Li_collector}), we deduce from equation \eqref{ionic_conservation_n}, computed at the point $x_{cn}$, the boundary condition of the electrolyte potential at the anode collector
$$
\partial_x\phi_{en}(x_{cn},t)=0.
$$
Similarly, from the boundary conditions $i_{ep}(x_{cp},t)=0$  (see \eqref{no_ion_transfer_at_collector}) and  $\partial_x c_{ep}(x_{cp},t)=0$ (see \eqref{eq_boundary_condition_Li_collector}), we deduce, using equation \eqref{ionic_conservation_p} at the point $x_{cp}$, the boundary condition of the electrolyte potential at the cathode collector
$$
\partial_x\phi_{ep}(x_{cp},t)=0.
$$ 
Integration of the $j_{en}^{BV}$ conservation equation  \eqref{cation_density_current_conservation} into $\Omega_n$ yields
$$
i_{en}(x_{en},t)-i_{en}(x_{cn},t)=\int_{x_{cn}}^{x_{en}} \jBV{en}(x,t)\,dx \implies A\,i_{en}(x_{en},t)=A\,L_n \ovjBV{en}(t)=+I(t).
$$
Similarly, the $i_{ep}$ conservation equation   \eqref{cation_density_current_conservation} over $\Omega_p$ yields
$$
A\,i_{ep}(x_{ep},t)=-A\,L_p \ovjBV{ep}(t)=+I(t).
$$ 
Integrating the conservative equation    $\partial_x i_{es}=0$ (see \eqref{cation_density_current_conservation}) over $\Omega_s$ provides the flux conservation
$$
i_{es}(x_{en},t)=i_{es}(x_{ep},t).
$$
 Evaluating the relation \eqref{ionic_conservation_s} at the point $x_{en}$ and $x_{ep}$ gives with 
\begin{eqnarray*}
i_{es}(x_{en},t)&=&-\sigma_{es}\left [
\partial_x \phi_{es}(x_{en},t) +\frac{(2t^+-1)}{f}\; \frac{\partial_x c_{es}(x_{en},t)}{c_{es}(x_{en},t)}
\right ],\\
i_{es}(x_{ep},t)&=&-\sigma_{es}\left [
\partial_x \phi_{es}(x_{ep},t) +\frac{(2t^+-1)}{f}\; \frac{\partial_x c_{es}(x_{ep},t)}{c_{es}(x_{ep},t)}
\right ],
\end{eqnarray*}
with $f=\frac{F}{RT}$ $[V^{-1}]$ $=[CJ^{-1}]$.

\section{Towards the 0D model}
The average values of physical quantities, together with pointwise values at the interfaces, lead to a differential-algebraic system coupling average values over the subdomains of current densities, concentrations, or electric potentials at the interfaces. At that stage, we have just derived new relations using strictly mathematical procedures, and no additional assumptions have been introduced.
Nevertheless, the system is not closed since no specific information about the shape of the physical variables in space has been defined. To elaborate on the 0D model, we shall prescribe some acceptable approximations in space (linear or quadratic shape), enabling us to close the algebraic-differential system. 
\subsection{Some mathematical properties about the approximations}
We first recall some properties of general approximations. Let $\Lambda=[x_1,x_2]$ be an interval and $g(x)=f(u(x))$ where $f$ and $u$ are real functions that are smooth enough in $\Lambda$. We recall that the average is given by $\displaystyle \bar u=\frac{1}{x_2-x_1} \int_{x_1}^{x_2} u(x)\,dx$  and similarly $\bar g$ the average value of $u$ and $g$, respectively, over the domain $\Lambda$, while $\displaystyle m=\frac{x_1+x_2}{2}$ is the midpoint of $\Lambda$. We have the following properties.
\begin{itemize}
\item Second-order approximation of average and midpoint values.
$$
\overline{u}=u(m)+O\big( (x_2-x_1)^2 \big),\qquad \overline{g}= f(\overline{u})+O\big( (x_2-x_1)^2 \big).
$$
 Note that if $f$ is affine, we have the equality $\overline{g}= f(\overline{u})$. 
\item Second-order approximation of a product: for any regular functions $u$ and $v$, one has
$$
\overline{uv}=(uv)(m)+O\big( (x_2-x_1)^2 \big)= \overline{u}\, \overline{v}+O\big( (x_2-x_1)^2 \big).
$$
\end{itemize}
\subsection{The main hypothesis: the Butler-Volmer average}\label{sec:approx:BV}
The principal assumption we make concerns the Butler-Volmer equation. We assume that the rate of production/consumption over the domains $\Omega_k$, $k\in \{n,p\}$,  is approximated by its average to second order in space. This assumption is admissible if the variations of $\jBV{ek}(x,t)$ are small with respect to the average values. Consequently, we state
$$
\jBV{ek}(x,t)= \ovjBV{ek}(t)+O\big (L_k^2\big ).
$$

The second hypothesis we state is the substitution of the average of a function by the function of the averages, namely (see section \ref{sec::BVsusbsec})
$$
\ovjBV{ek}=\frac{1}{L_k}\int_{\Omega_k} \chi_{ek}(\eta_{ek})\cdot j^0_{ek}(c_{ek},c_k)\, dx=
\chi_{ek}(\oveta_{ek})\cdot j^0_{ek}(\ovc_{ek},\ovc_k)+O\big(L_k^2\big ).
$$
In conclusion, we substitute the original Butler-Volmer relation with the approximation
\begin{equation}\label{asumption_BV}
\ovjBV{ek}=\chi_{ek}(\oveta_{ek})\cdot j^0_{ek}(\ovc_{ek},\ovc_k).
\end{equation}

\subsection{Approximation of the current density in the electrodes}
Since we assume that the Butler-Volmer source term is approximated by a constant function in space $\ovjBV{ek}$(t), the electronic density current $i_n(x,t)$ satisfies $\partial_x i_n=-\ovjBV{en} (t)$ (see \eqref{eq::electronic_charge_conservation}) and suggests a linear approximation in space. Taking advantage of the boundary condition $i_n(x_{en},t)=0$ and  $i_{n}(x_{cn},t)=L_n\ovjBV{en}(t)$  (see  \eqref{eq::electronic_flux_interface}  and \eqref{eq::cutent_density_cn})), we deduce
$$
i_n(x,t)=  (x_{en}-x) \ovjBV{en}(t).
$$
Similarly, the density current for the positive electrode is now given by $\partial_x i_p(x,t)= -\ovjBV{ep}(t)$  and, with the boundary condition $i_p(x_{ep},t)=0$,  (see \eqref{eq::electronic_flux_interface}) and  $i_p(x_{cp},t)=-L_p\ovjBV{ep}(t)$ (see \eqref{eq::cutent_density_cp}), we deduce the approximation
$$
i_p(x,t)= (x_{ep}-x) \ovjBV{ep}(t).
$$

\subsection{Approximation of the electronic potential in the electrodes}
Since the current density is a linear function in space, we deduce that the electric potential is a quadratic function.
Using the approximation of the Butler-Volmer term and the potential in the collector of the negative electrode given by \eqref{eq::elec_potential_anode}, 
$$
\Phi_{cn}(t)=\phi_n(x_{en},t)+ \frac{1}{\sigma_n}\ovjBV{en}(t) \int_{x_{cn}}^{x_{en}} (s-x_{cn})\,ds,
$$
and, by integration, provides the equation for the anode potential
$$
\Phi_{cn}(t)=\phi_n(x_{en},t)+\frac{(L_n)^2}{2\sigma_n} \ovjBV{en}(t).
$$
On the other hand, a second-order average approximation of the negative electrode potential is given by the trapezoidal formula 
$$
\ovphi_{n}(t)= \frac{1}{2}\big (\phi_n(x_{en},t)+\phi_{n}(x_{cn},t) \big )
$$
and we deduce using the relation \eqref{eq::average_intensity_n} 
\begin{equation}\label{eq::average_potential_anode}
\ovphi_{n}(t)=\Phi_{cn}(t)-\frac{L_n}{4A\,\sigma_n}I(t).    
\end{equation}

Similarly, on the positive electrode, from the relation \eqref{eq::elec_potential_catode}, we deduce the approximation
$$
\Phi_{cp}(t)=\phi_p(x_{ep},t)+\frac{(L_p)^2}{2\sigma_p} \ovjBV{ep}(t).
$$
Using once again the trapezoidal rule together with \eqref{potencial_interface_collector} and \eqref{eq::average_intensity_p}, we give an approximation of the average of the electric potential in the positive electrode
\begin{equation}\label{eq::average_potential_cathode}
\ovphi_{p}(t)=\Phi_{cp}(t)+\frac{L_p}{4A\,\sigma_p} I(t).
\end{equation}
\begin{remark}
Note that the potential difference $\Delta_{elec}=\Phi_{cp}-\Phi_{cn}$ corresponds to the voltage between the anode and the cell cathode. \qedsymbol
\end{remark}

\subsection{Approximation of the ionic density current in the electrolyte}
Since we assume that $\jBV{ek}(x,t)=\ovjBV{ek}(t)$ is constant in space, from \eqref{cation_density_current_conservation} and applying \eqref{eq::average_intensity_n}, the ionic current is given by the equation
$$
\partial_x i_{en}(x,t) = \frac{I(t)}{A\,L_n}.
$$
Therefore, $i_{en}(x,t)$ is linear in space. Using the boundary condition  \eqref{no_ion_transfer_at_collector}, we get the expression
\begin{equation}\label{jen_linear_space}
 i_{en}(x,t)=(x-x_{cn})\frac{I(t)}{A\,L_n}. 
\end{equation}
We apply a similar procedure for the ionic current density in the positive electrode and, from  \eqref{eq::average_intensity_p} together with the boundary condition  $i_{ep}(x_{cp},t)=0$, (see \eqref{no_ion_transfer_at_collector}) we deduce the linear approximation
\begin{equation}\label{jep_linear_space}
i_{ep}(x,t)=(x_{cp}-x)\frac{I(t)}{A\,L_p}.
\end{equation}
Finally, the current conservation equation $\partial_x i_{es}=0$ in $\Omega_s$ (see \eqref{cation_density_current_conservation})  together with the boundary conditions \eqref{ionic_flux_continuity_separator}, $i_{es}(x_{en},t)=i_{en}(x_{en},t)$  indicates that the density current is approximated by constant values in space and leads to the relation
\begin{equation}\label{jes_constant_space}
i_{es}(x,t)= \frac{I(t)}{A}.
\end{equation}

\subsection{Approximation of the  potential in the electrolyte}
\subsubsection{Electrolyte potential in the electrodes}
We now derive the approximation of the electrolyte potential in the subdomains by taking advantage of the previous stages. The construction is achieved in three steps.

\hskip 1em $\bullet$ {\sc step 1.}
We recall the original equation  \eqref{ionic_conservation_n} for the ionic density current
$$ 
i_{en}(x,t)=-\sigma_{en}\left [\partial_x \phi_{en}(x,t)+\frac{(2t^+-1)}{f}\partial_x \ln(c_{en}(x,t))\right ].
$$
Since $i_{en}$ is a linear function in space given by \eqref{jen_linear_space}, we get 
$$
(x-x_{cn})\frac{I(t)}{A\sigma_{en}L_n}= -\partial_x \phi_{en}(x,t)-
\frac{(2t^+-1)}{f}\partial_x \ln(c_{en}(x,t)).
$$

\hskip 1em $\bullet$ {\sc step 2.} Integrating over $[x_{cn},x]$, we get
\begin{equation}\label{eq::potential_quadratic}
\phi_{en}(x,t)= \phi_{en}(x_{cn},t)-\frac{(x-x_{cn})^2}{2L_n} 
\frac{I(t)}{A\sigma_{en}}-\frac{(2t^+-1)}{f}\ln\left (\frac{c_{en}(x,t)}{c_{en}(x_{cn},t}\right )
.    
\end{equation}
In particular, for $x=x_{en}$ and denoting $\Phi_{en}(t)=\phi_{en}(x_{en},t)$, we have the ionic potential at the interface between the negative electrode and the separator given by
$$
\Phi_{en}(t)=\phi_{en}(x_{cn},t)-\frac{L_n}{2}\frac{I(t)}{A\sigma_{en}}
-\frac{(2t^+-1)}{f}\ln \left (\frac{c_{en}(x_{en},t)}{c_{en}(x_{cn},t)}\right ).
$$

\hskip 1em $\bullet$ {\sc step 3.}
Finally, we compute an approximation of the average potential using the trapezoidal rule 
$$\displaystyle \ovphi_{en}(t)= \frac{1}{2}\big (\Phi_{en}(t)+\phi_{en}(x_{cn},t)\big )$$ 
and we get the following
\begin{equation}\label{eq::potential_drop_anode}
\Delta_{en}=\ovphi_{en}(t)-\Phi_{en}(t)=\frac{L_n}{4} \frac{I(t)}{A\sigma_{en}}
+\frac{(2t^+-1)}{2f}\ln \left (\frac{c_{en}(x_{en},t)}{c_{en}(x_{cn},t)}\right ).    
\end{equation}

\vskip 1em
We proceed in a similar way in the positive electrode using the original equation \eqref{ionic_conservation_p} and the linear approximation of $i_{ep}$ in space  \eqref{jep_linear_space}. Integration over $[x,x_{cp}]$ gives
$$
\phi_{ep}(x,t)=\phi_{ep}(x_{cp},t)+\frac{(x-x_{cp})^2}{2L_p}\frac{I(t)}{A\,\sigma_{en}} 
+\frac{(2t^+-1)}{f}\ln \left (\frac{c_{ep}(x_{cp},t)}{c_{ep}(x,t)}\right ).
$$

At last, using the trapezoidal rule for the average approximation and setting $\Phi_{ep}(t)=\phi_{ep}(x_{ep},t)$, we get 
\begin{equation}\label{eq::potential_drop_cathode}
\Delta_{ep}=\ovphi_{ep}(t)-\Phi_{ep}(t)=-\frac{L_p}{4}\frac{I(t)}{A\sigma_{ep}}
-\frac{(2t^+-1)}{2f}\ln \left (\frac{c_{ep}(x_{cp},t)}{c_{ep}(x_{ep},t)}\right ).
\end{equation}

\subsubsection{Electrolyte potential in the separator}
From the initial relation  \eqref{ionic_conservation_s},  together with relation  \eqref{jes_constant_space} $\displaystyle i_{es}(x,t)=\frac{I(t)}{A}$, we deduce
\begin{equation}\label{potential_intensity}
 \frac{I(t)}{A\sigma_{es}}= -\partial_x \phi_{es}(x,t)-\frac{(2t^+-1)}{f}\partial_x \ln (c_{es}(x,t)).   
\end{equation}
Integration over $[x_{en},x]$ provides the relation
$$
\phi_{es}(x,t)=\phi_{es}(x_{en},t)-(x-x_{en})\frac{I(t)}{A\sigma_{es}}-\frac{(2t^+-1)}{f}\ln \left (\frac{c_{es}(x,t)}{c_{es}(x_{en},t)}\right ).
$$
We have assumed that the continuity of the potential at the separator interface yields $\phi_{es}(x_{en},t)=\phi_{en}(x_{en},t)=\Phi_{en}(t)$ and $\phi_{es}(x_{ep},t)=\phi_{ep}(x_{ep},t)=\Phi_{ep}(t)$.
We use the last relation at point $x=x_{ep}$ and finally, we get the relation
\begin{equation}\label{eq::potential_drop_separator}
\Delta_{es}=\Phi_{ep}(t)-\Phi_{en}(t)=-\frac{L_s}{A\sigma_{es}}I(t)
-\frac{(2t^+-1)}{f}\ln \left (\frac{c_{es}(x_{ep},t)}{c_{es}(x_{en},t)}\right ).
\end{equation}

\subsection{Approximation of \ce{Li+} concentration}
Constructing an accurate approximation of the \ce{Li+} concentration is critical for two reasons. On the one hand, we have to assess $\ovc_{en}$ and $\ovc_{ep}$ to compute the Butler-Volmer contribution and, on the other hand, we need the \ce{Li^+} concentrations at the interfaces $x_{cn}$, $x_{en}$, $x_{ep}$, and $x_{cp}$ to evaluate $\Delta_{en}$, $\Delta_{es}$, and $\Delta_{ep}$.

An accurate approach will consist of solving the full parabolic equation, both in time and space. To provide a fully 0D model, we state that the ion concentrations are almost governed by a time-parametrized steady-state equation; thus the time derivative in equations \eqref{eq::conservation_Lin}, \eqref{eq::conservation_Lip}, \eqref{eq::conservation_Lis} is cancelled. Then we get with  \eqref{eq::average_intensity_n} and \eqref{eq::average_intensity_p}, the new system
\begin{eqnarray}
-\partial_x \biggl(D_{en} \partial_x c_{en} \biggl)&=&\displaystyle  \frac{(1-t^+)}{F} \jBV{en}=\frac{(1-t^+)}{A\,F\,L_n}I(t)\quad x\in \Omega_n, \label{eq::conservation_Lin_steady}\\
-\partial_x \biggl(D_{ep} \partial_x c_{ep} \biggl)&=&\displaystyle  \frac{(1-t^+)}{F} \jBV{ep}=-\frac{(1-t^+)}{A\,F\,L_p}I(t)\quad x\in \Omega_p, \label{eq::conservation_Lip_steady}\\
-\partial_x \biggl(D_{es} \partial_x c_{es} \biggl)&=&0\quad x\in \Omega_s, \label{eq::conservation_Lis_steady}
\end{eqnarray}
\begin{remark}
Cancelling the time derivative does not mean that the problem is time-independent. We just say that the concentrations of the different species vary sufficiently slowly to consider that the term $\partial_t c_e$ is not significant with respect to the other contributions. However, the parameters, such as the source term, may still depend on time. \qedsymbol
\end{remark}

Let $\displaystyle C(t)=\frac{(1-t^+)}{AF}I(t)$, integrating relations \eqref{eq::conservation_Lin_steady} in $[x_{cn},x]$ and \eqref{eq::conservation_Lip_steady} in $[x,x_{cp}]$, together with the homogeneous flux condition at the collectors, {\it i.e.,} $\partial_x c_{en}(x_{cn},t)=0$ and $\partial_x c_{ep}(x_{cp},t)=0$, give the following
\begin{eqnarray*}
D_{en}\partial_x c_{en}(x,t) &=&-\frac{(x-x_{cn})}{L_n}C(t)\quad x\in \Omega_n,\\
D_{ep}\partial_x c_{ep}(x,t)&=&+\frac{(x-x_{cp})}{L_p}C(t)\quad x\in \Omega_p, \\
D_{es}\partial_x c_{es}(x,t)&=&D_{es}\partial_x c_{es}(x_m,t) \quad x\in \Omega_s. 
\end{eqnarray*}
with $x_m$ the midpoint of $[x_{en},x_{ep}]$. Flux continuity provides the compatibility condition of the ionic flux.
$$
D_{ep} \partial_x c_{ep}(x_{ep},t)=D_{en} \partial_x c_{en}(x_{en},t)=D_{es}\partial_x c_{es}(x_m,t) =-C(t).
$$

Integrating  equations \eqref{eq::conservation_Lin_steady}, \eqref{eq::conservation_Lis_steady}, \eqref{eq::conservation_Lip_steady} twice in space, we obtain the analytical expressions for the three functions 
\begin{equation}
\left \{ \begin{array}{l}
c_{en}(x,t)= c_{en,c}(t)-\displaystyle C(t)\frac{(x-x_{cn})^2}{2D_{en}L_n}\quad x\in \Omega_n,\\
c_{ep}(x,t)= c_{ep,c}(t)+\displaystyle C(t)\frac{(x-x_{cp})^2}{2L_pD_{ep}}\quad x\in \Omega_p,\\
c_{es}(x,t)= c_{es,m}(t)-\displaystyle C(t)\frac{(x-x_{m})}{D_{es}}\quad x\in \Omega_s,
\end{array}\right . \label{eq::concentration_0D}
\end{equation}
 with  $c_{en,c}(t)=c_{en}(x_{cn},t)$, $c_{es,m}(t)=c_{es}(x_m,t)$, $c_{ep,c}(t)=c_{ep}(x_{cp},t)$. 

The continuity in $x_{en}$, $x_{ep}$ (see \eqref{eq_boundary_condition_Li_s}) reads
$$
c_{en,c}(t)-C(t)\frac{L_n}{2D_{en}}=c_{es,m}(t)+C(t)\frac{L_s}{2D_{es}},\qquad
c_{ep,c}(t)+C(t)\frac{L_p}{2D_{ep}}=c_{es,m}(t)-C(t)\frac{L_s}{2D_{es}}.
$$
On the other hand, the average values are given by 
\begin{equation}\label{eq::constant_concentration}
\ovc_{en}(t)=c_{en,c}(t)-C(t)\frac{L_n}{6D_{en}},\quad  \ovc_{es}(t)=c_{es,m}(t),\quad \ovc_{ep}(t)=c_{ep,c}(t)+C(t)\frac{L_p}{6D_{ep}}
\end{equation}
and, using the previous relations, we deduce the continuity relations, 
$$
\ovc_{en}(t)-C(t)\frac{L_n}{3D_{en}}=\ovc_{es}(t)+C(t)\frac{L_s}{2D_{es}},\qquad \ovc_{ep}(t)+C(t)\frac{L_p}{3D_{ep}}=\ovc_{es}(t)-C(t)\frac{L_s}{2D_{es}}.
$$
Finally, mass conservation of the \ce{Li^+} \eqref{eq_conservation_Li+} reads
$$
\varepsilon_n L_n \ovc_{en}+\varepsilon_s L_s \ovc_{es}+\varepsilon_p L_p \ovc_{ep}=(\varepsilon_n L_n +\varepsilon_s L_s +\varepsilon_p L_p)\,c_{e}^0=M^0
$$ 
with $c_{e}^0$ denoting  the   initial concentration of \ce{Li^+}. In conclusion, the average concentrations in the dielectric are given by the linear system
\begin{equation}
\left \{\begin{array}{l}
\displaystyle \ovc_{en}(t)=\ovc_{es}(t)+C(t)\left ( \frac{L_n}{3D_{en}}+\frac{L_s}{2D_{es}}\right ),\\[0.5em]
\displaystyle \ovc_{ep}(t)=\ovc_{es}(t)-C(t)\left ( \frac{L_p}{3D_{ep}}+\frac{L_s}{2D_{es}}\right ),\\[0.5em]
\displaystyle \varepsilon_n L_n \ovc_{en}(t)+\varepsilon_s L_s \ovc_{es}(t)+\varepsilon_p L_p \ovc_{ep}(t)=M^0.
\end{array}
\right .  \label{eq::average_concentration_0D}
\end{equation}

The linear system \eqref{eq::average_concentration_0D} shows that the time dependency is only related to the function $C(t)$, that is, the intensity $I(t)$. In conclusion, the average concentrations in the dielectric only explicitly depend on the intensity $I$ through the variable $C$. Hence, consider the new linear system using $(\hatc_{en},\hatc_{es},\hatc_{ep})$ as unknown functions regarded the intensity
\begin{equation}
\left \{\begin{array}{l}
\displaystyle \hatc_{en}(I)=\hatc_{es}(I)+\frac{(1-t^+)}{AF}\left ( \frac{L_n}{3D_{en}}+\frac{L_s}{2D_{es}}\right )I,\\[0.5em]
\displaystyle \hatc_{ep}(I)=\hatc_{es}(I)-\frac{(1-t^+)}{AF}\left ( \frac{L_p}{3D_{ep}}+\frac{L_s}{2D_{es}}\right )I,\\[0.5em]
\displaystyle \varepsilon_n L_n \hatc_{en}(I)+\varepsilon_s L_s \hatc_{es}(I)+\varepsilon_p L_p \hatc_{ep}(I)=M^0.
\end{array}
\right .  \label{eq::average_concentration_0D_hat}
\end{equation}
We deduce that
\begin{equation}\label{eq::average_concentration_hat_defn}
   \hatc_{en}(I(t))\equiv\ovc_{en}(t), \quad  \hatc_{ep}(I(t))\equiv\ovc_{ep}(t), \quad  \hatc_{es}(I(t))\equiv\ovc_{es}(t).
\end{equation}

Analogously, we find that concentrations at the boundaries, $c_{en,c}$, $c_{ep,c}$, $c_{es,m}$  also explicitly depend directly on the intensity $I$ and and from \eqref{eq::constant_concentration} we have the following 
\begin{equation}\label{eq::constant_concentration_hat}
\hatc_{en,c}(I)=\hatc_{en}(I)+\frac{(1-t^+)}{AF}\frac{L_n}{6D_{en}}I,\quad  
\hatc_{es,m}(I)=\hatc_{es}(I),\quad 
\hatc_{ep,c}(I)=\hatc_{ep}(I)-\frac{(1-t^+)}{AF}\frac{L_p}{6D_{ep}}I.
\end{equation}
and satisfy the chain-rule
\begin{equation}\label{eq::constant_concentration_hat_defn}
   \hatc_{en,c}(I(t))\equiv c_{en,c}(t), \quad  \hatc_{ep,c}(I(t))\equiv c_{ep,c}(t), \quad  \hatc_{es,m}(I(t))\equiv c_{es,m}(t),
\end{equation}
\begin{remark}\label{rem_I-linearity-boundary}
By construction, the average approximation functions $\hatc_{en}(I)$, $\hatc_{es}(I)$, $\hatc_{ep}(I)$ together with the values at the interface with the collectors $\hatc_{en,c}(I)$ and $\hatc_{en,c}(I)$ are linear in terms of intensity $I$.
\end{remark}
\subsection{Overpotential and Butler-Volmer relation}
From section \ref{sec:approx:BV}, we approximate the average of the density current  $j^0_{ek}$, $k \in \{n,p\}$, transfer equation \eqref{eq::exchange_density_current}, with an approximation only depending on the concentrations' average and overpotentials. We write it in the form 
\begin{eqnarray*}
\ovj_{en}^0(t)&=&j^{0*}_{en}\, \left ( \frac{\ovc_n(t)}{c_n^{*}}\right )^{\alpha_n}\left ( \frac{c_n^{max}-\ovc_n(t)}{c_n^{max}-c_n^{*}}\right )^{1-\alpha_{n}}\left ( \frac{\ovc_{en}(t)}{c_{en}^{*}}\right )^{1-\alpha_{n}},\\
\ovj_{ep}^0(t)&=&j^{0*}_{ep}\, \left ( \frac{\ovc_p(t)}{c_p^{*}}\right )^{\alpha_p}\left ( \frac{c_p^{max}-\ovc_p(t)}{c_p^{max}-c_p^{*}}\right )^{1-\alpha_{p}}\left ( \frac{\ovc_{ep}(t)}{c_{ep}^{*}}\right )^{1-\alpha_{p}}.
\end{eqnarray*}

On the other hand, the Butler-Volmer source term is assumed to be constant in the porous electrode (see section \ref{sec:approx:BV}); then we have the following
\begin{eqnarray*}
\ovjBV{en}(t)&=&\ovj_{en}^0(t)\cdot\chi_{en} \big(\oveta_{en}(t)\big),\\
\ovjBV{ep}(t)&=&\ovj_{ep}^0(t)\cdot\chi_{ep} \big(\oveta_{ep}(t)\big).
\end{eqnarray*}
Using the special function $\bsinh(\eta;\beta)$, we deduce from \eqref{defn_chi}
$$
\chi_{en} \big(\oveta_{en}(t)\big)=2\bsinh(f\oveta_{en}(t)/2;\alpha_n),\quad
\chi_{ep} \big(\oveta_{ep}(t)\big)=2\bsinh(f\oveta_{ep}(t)/2;\alpha_p).
$$
Furthermore, we rely on the electric and ionic potentials via the over-potential definition, that is \eqref{defn_overpotential},
$$
\oveta_{en}(t)=\ovphi_{n}(t)-\ovphi_{en}(t)-U_{en},\qquad 
\oveta_{ep}(t)=\ovphi_{p}(t)-\ovphi_{ep}(t)-U_{ep}.
$$

Finally, the positive and negative electrodes are connected via the external circuit
\begin{equation} \label{eq::ext_circuit}
\Phi_{cp}-\Phi_{cn}=R_{ext} I+U_{ext}    
\end{equation}
with $R_{ext}$ $[\Omega]$ the user's external resistance and $U_{ext}$ $[V]$ the external potential (for instance, when charging the cell).

\section{The time-parametrized model}
In the previous section, we introduced a set of space-independent variables together with relations deriving from the original problem. This section is devoted to the construction of the 0D model by coupling the different sub-systems.   

\subsection{The concentration formulation}\label{subsec::time_conce_form}

Given the intensity $I=I(t)$, we deduce the concentration in the electrodes from \eqref{eq::anode_species_conc}-\eqref{eq::catode_species_conc} and the relations \eqref{eq::average_intensity_n}-\eqref{eq::average_intensity_p}, 
\begin{equation}\label{eq::species_conc}
\varepsilon_n\,\frac{d\ovc_{n}}{dt}(t)=-\frac{I(t)}{A\,F \,L_n},\qquad
\varepsilon_p \,\frac{d\ovc_{p}}{dt}(t)=+\frac{I(t)}{A\,F\,L_p}.   
\end{equation}
On the other hand, from relation \eqref{eq::average_concentration_0D}, we deduce $\ovc_{es}(t)$ with 
\begin{equation}\label{continuos:separator}
(\varepsilon_n L_n +\varepsilon_s L_s +\varepsilon_p L_p)\,\ovc_{es}(t)=M^0
+C(t) \left \{ \varepsilon_p L_p\left (\frac{L_p}{3D_{ep}}+\frac{L_s}{2D_{es}} \right )
-\varepsilon_n L_n\left (\frac{L_n}{3D_{en}}+\frac{L_s}{2D_{es}}  \right )
\right\}
\end{equation}
and then the two other average values, $\ovc_{en}, \ovc_{ep}$, using the linear system \eqref{eq::average_concentration_0D}
\begin{equation}\label{eq::electrolyte_conc}
\ovc_{en}(t)=\ovc_{es}(t)+C(t)\left (\frac{L_n}{3D_{en}}+\frac{L_s}{2D_{es}}\right ),\qquad 
\ovc_{ep}(t)=\ovc_{es}(t)-C(t)\left (\frac{L_p}{3D_{ep}}+\frac{L_s}{2D_{es}}\right ).  
\end{equation}

Moreover, we have at the interfaces 
\begin{equation}\label{eq::electrolyte_bound}
 \begin{aligned}
c_{en}(x_{cn},t)=\ovc_{en}(t) +  C(t)\frac{L_n}{6D_{en}},&\qquad 
c_{es}(x_{en},t)=\ovc_{es}(t) +  C(t)\frac{L_s}{2D_{es}},\\ 
c_{es}(x_{ep},t)=\ovc_{es}(t) -  C(t)\frac{L_s}{2D_{es}},&\qquad 
c_{ep}(x_{cp},t)=\ovc_{ep}(t) -  C(t)\frac{L_p}{6D_{ep}}.
\end{aligned}   
\end{equation}

\subsection{The potential formulation}
Assuming that the average concentrations are known, we determine the current density and the potential by solving a time-parametrized system. We decompose the potential difference $\Phi_{cp}-\Phi_{cn}$ between the two electrodes with
\begin{eqnarray*}
\Phi_{cp}-\Phi_{cn}&=&(\Phi_{cp}-\ovphi_p)+(\ovphi_p-\ovphi_{ep})+(\ovphi_{ep}-\Phi_{ep})+(\Phi_{ep}-\Phi_{en})+(\Phi_{en}-\ovphi_{en})- (\ovphi_{en}-\ovphi_n)+(\ovphi_n-\Phi_{cn}).
\end{eqnarray*}
Using the relations \eqref{defn_overpotential}, \eqref{eq::potential_drop_anode}, \eqref{eq::potential_drop_cathode}, \eqref{eq::potential_drop_separator}, \eqref{eq::average_potential_anode} and \eqref{eq::average_potential_cathode}, we rewrite the decomposition in the following form
\begin{eqnarray*}
\Phi_{cp}-\Phi_{cn}&=& U_{ep}+\eta_{ep}-U_{en}-\eta_{en}-R_{int}I-\Delta_{c}
\end{eqnarray*}
where
$$
R_{int}= \left (\frac{L_p}{4A\,\sigma_p} +\frac{L_p}{4A\sigma_{ep}}+\frac{L_s}{A\sigma_{es}}+\frac{L_n}{4A\sigma_{en}}+ \frac{L_n}{4A\,\sigma_n} \right )
$$
stands for the internal resistance of the cell, and
\begin{equation}\label{eq::exact_DELTA_c}
\Delta_c=\frac{2t^+-1}{2f}
\left [
\ln \left ( \frac{c_{ep}(x_{cp},t)}{c_{es}(x_{ep},t)}\right ) +
2\ln \left (\frac{c_{es}(x_{ep},t)}{c_{es}(x_{en},t)}\right )  +
\ln \left ( \frac{c_{es}(x_{en},t)}{c_{en}(x_{cn},t)}\right )
\right ]
\end{equation}
is the potential drop due to the gradient of the ionic concentration we rewrite
$$
\Delta_c=\frac{2t^+-1}{2f}
\ln \left ( \frac{c_{ep}(x_{cp},t)\,c_{es}(x_{ep},t)}{c_{es}(x_{en},t)\,c_{en}(x_{cn},t)}\right ).
$$

The potential decomposition now reads
\begin{eqnarray*}
\ \Phi_{cp}-\Phi_{cn}&=&U_{ep}+\eta_{ep}-U_{en}-\eta_{en}-R_{int}I-\Delta_{c},\\ 
I&=&+\iota_n(\ovc_n,\ovc_{en})\ \bsinh(f\eta_{en}/2),\\
I&=&-\iota_p(\ovc_p,\ovc_{ep})\ \bsinh(f\eta_{ep}/2),  
\end{eqnarray*}
where, (see \eqref{eq::average_intensity_n}, \eqref{eq::average_intensity_p}),
\begin{eqnarray}
\iota_n(\ovc_n,\ovc_{en})&=&2F\,A\, L_n\,a_n\,r_n\, \big (\ovc_{en}\big )^{\alpha_n} \big (\ovc_n\,(c^{max}_n-\ovc_n)\big )^{1-\alpha_n},\label{def_iota_n}\\ 
\iota_p(\ovc_p,\ovc_{ep})&=&2F\,A\, L_p\,a_p\,r_p\, \big (\ovc_{ep}\big )^{\alpha_p} \big (\ovc_p\,(c^{max}_p-\ovc_p)\big )^{1-\alpha_p}.\label{def_iota_p}
\end{eqnarray}
On the other hand, we have to fulfill the relation $\Phi_{cp}-\Phi_{cn}=R_{ext}I+U_{ext}$ which corresponds to the external circuit, (see \eqref{eq::ext_circuit}).

Equating the potential difference between the internal and external circuits provides the non-linear equation  for the intensity
$$
U_{ep}-U_{en}-U_{ext}-\Delta_c=g(I,R_{ext}),
$$
with 
$$
g(I,R_{ext})=(R_{ext}+R_{int})I+\frac{2}{f}\left [ \absinh\left ( \frac{I}{\iota_n}\right)+\absinh\left ( \frac{I}{\iota_p}\right)\right ],
$$
where $\iota_n$ and $\iota_p$ are non-negative functions. 
\begin{remark}
For the particular case $\alpha_n= \alpha_p=1/2$, we recover the standard hyperbolic functions, and the relation reads
$$
g(I,R_{ext})=(R_{ext}+R_{int})I+\frac{2}{f} \left [ \arg\sinh\left ( \frac{I}{\iota_n}\right)+\arg\sinh\left ( \frac{I}{\iota_p}\right) \right ].
\hskip3em \qedsymbol
$$
\end{remark}
Let us define the residual potential
\begin{equation}\label{eq:residual_potential}
V_{res}(I,R_{ext})=U_{ep}-U_{en}-U_{ext}-\Delta_c-g(I,R_{ext}).
\end{equation}
The solutions of the cell problem are given by the solution of $V_{res}(I,R_{ext})=0$, that is, we seek $I$ such that $V_{res} = 0$ given the user parameters $R_{ext}$ and $U_{ext}$ and the averaged concentrations.

In conclusion, function $V_{res}(I,R_{ext})$  has the following properties:
\begin{itemize}
\item  is continuous, linear in $R_{ext}$, differentiable in $I$;
\item  is a strictly decreasing function in order to $I$ since $R_{ext}\geq 0$,
$$
\lim_{I\to -\infty}V_{res}(I,R_{ext})=+\infty,\qquad \lim_{I\to +\infty}V_{res}(I,R_{ext})=-\infty. 
$$
\end{itemize}
\begin{prop}\label{prop:existence_uniqueness}
For any $R_{ext}\geq 0$, there exists a unique intensity $\bar I$ such that  $V_{res}(\bar I,R_{ext})=0$.
\end{prop}
\subsection{A note about the critical situations}\label{sec::steady-state}
From a practical point of view, successive charge and discharge cell are carried out to assess the number of cycles that the device may support. To this end, a setup is built to perform the operation under a constant intensity and, in that way, define the number of hours for one cycle. In common language, practitioners deal with charge or discharge under {\bf constant intensity} $I_0$ or target intensity. However, constant intensity is not always feasible, particularly when the cell approaches a fully charged or fully discharged state. These critical situations occur when either $\iota_n$ or $\iota_p$ tends to zero because one of the relevant concentrations approaches zero or its maximum admissible value.\\
Consequently, the intensity has to decrease to zero when reaching the limit of the cell capacity. The intensity has to drop, even to zero, when the cell reaches an extreme state, {\it i.e.} $\iota_n=0$ or $\iota_p=0$. This is a critical issue from the modelling point of view, and the numerical method has to be designed to simulate the extreme situations. 
\begin{remark}
Function $g(I,R_{ext})$ also depends on other parameters, such as the average concentrations. We do not mention it explicitly for the sake of simplicity, but when the concentration dependency would be critical, we overload the notation with $g(I_0,R_{ext};\ovc_n,\ovc_{en},\ovc_p,\ovc_{ep})$. \qedsymbol
\end{remark}

In \cite{RFR22,DGR18} a similar analysis is provided on the behaviour of the current density close to the extreme cases.
Basically, reactive material concentrations in the anode, cathode, and electrolyte drive the electrical potential of the cell, while external resistance and electromotive force drive the intensity. The question is to assess whether the cell is in a condition to deliver (or receive) a target intensity $I_0$ prescribed by the user, that is, if it is possible to adjust the resistance $R_{ext}$ to reach the target intensity.

Extreme situations occur when $\iota_n(\ovc_n,\ovc_{en})$ or $\iota_p(\ovc_p,\ovc_{ep})$ go to zero. A non-null intensity  $I=I_0>0$ cannot be maintained. Indeed, assume for example that $\ovc_n\to 0$, then $\iota_n(\ovc_n,\ovc_{en})\to 0$ and $g(I_0,R_{ext};\ovc_n,\ovc_{en},\ovc_p,\ovc_{ep})\to +\infty$. In conclusion, even with $R_{ext}=0$, we no longer achieve the condition $V_{res}(I_0;0)=0$. Consequently, $I$ has to converge to zero when we reach a critical situation.

\subsection{Charge and discharge computations}
The early-stage analysis consists of determining whether we have a charging or discharging situation, given the open-circuit potentials $U_{en}$, $U_{ep}$, the external counter-electromotive potential $U_{ext}$ (motor or electrical source), and the potential $\Delta_c$ due to the \ce{Li^+} gradient. To this end, we first compute $V_{00}=V_{res}(0,0)$ corresponding to the open circuit. We have a charging situation if $V_{00}<0$ and a discharging situation if $V_{00}>0$. In the charging configuration, the target intensity $I_0<0$ is mandatory, while $I_0>0$  in the discharging configuration.

\subsubsection{Discharging the cell}
Since $V_{00}>0$, the second stage is to determine whether the target intensity $I_0>0$ is achievable, so we compute $V_0=V_{res}(I_0;0)$ and reach the following alternatives
\begin{itemize}
\item if $V_0\geq 0$, the external circuit equipped with resistance $\displaystyle R_{ext}=\frac{V_0}{I_0}$ satisfies the equilibrium $V_{res}(I_0;R_{ext})=0$ by construction. 
\item In contrast, $V_0<0$ shows that the target intensity is too high, and the unique solution of $V_{res}(\bar I;0)=0$ gives a lower intensity $\bar I\geq 0$. In particular, if one of the two components $\iota_n(\ovc_n,\ovc_{en})$ or $\iota_p(\ovc_p,\ovc_{ep})$ goes to zero, intensity $\bar I$ must go to zero, but the quotient with the $\arg\sinh$ denominator converges to a limit. 
\end{itemize}

\subsubsection{Charging the cell}
We proceed with the case $V_{00}<0$ where we charge the cell by applying an additional potential $U_{ext}$ of the battery charger.
\begin{itemize}
\item The target intensity $I_0<0$ is achievable if the residual potential $V_0=V_{res}(I_0;0)\leq 0$. In that case, the external resistance is given by $\displaystyle R_{ext}=\frac{V_0}{I_0}$.
\item  In contrast, if $V_0\geq 0$, intensity $I=I_0<0$ is not possible and we determine $\bar I<0$ by solving the non-linear problem $V_{res}(\bar I;0)=0$ to recover the equilibrium.
\end{itemize}

\section{Numerical Methods}
We shall consider two distinct kinds of simulation. On the one hand, we use the steady-state model where, given the concentrations, we compute the potentials and deduce the intensity; in particular, we assess whether we reach the target intensity or a lower intensity with a null external resistance. On the other hand, we consider the non-stationary case (charge and discharge) where the simulations are carried out up to a final time $\tend$. In that case, the experimental time interval $[0,\tend]$ is uniformly divided into $K+1$ sub-steps $t^k=k\Delta t$, $k=0,\cdots, K$, with $\displaystyle \Delta t=\tend/K$. The concentrations $\ovc(t)$ are approximated by $\ovc[k]\approx c(t^k)$ at time $t^k$. 

\subsection{Numerical solution for the concentrations}\label{sec:numerical_concentration_problem}
We construct numerical solutions for the concentration species assuming that the intensity is given (see subsection \ref{subsec::time_conce_form}). Let $I[k]=I(t^k)$, the given function defined in the interval $[0,\tend]$, and look for the average concentrations $\ovc_{n}[k]$, $\ovc_{p}[k]$, $\ovc_{en}[k]$, $\ovc_{es}[k]$, $\ovc_{ep}[k]$. 

\subsubsection{Computation of the concentrations}
We approximate the concentrations in the electrode with a simple forward Euler scheme for $\ovc_{n}$ {(see \eqref{eq::species_conc})
\begin{equation}\label{discrete:electrode}
\ovc_{n}[k]=\ovc_{n}[k-1]-\Delta t\, \frac{I[k]}{A\,F\varepsilon_n \,L_n},\qquad
\ovc_{p}[k]=\ovc_{p}[k-1]+\Delta t\, \frac{I[k]}{A\,F\varepsilon_p \,L_p},
\end{equation}
with $\ovc_{n}[0]$, $\ovc_{p}[0]$ given by the initial condition.

On the other hand, the electrolyte concentration in the separator is given by (see \eqref{continuos:separator})
\begin{equation}\label{discrete:separator}
(\varepsilon_n L_n +\varepsilon_s L_s +\varepsilon_p L_p)\,\ovc_{es}[k]=M^0+C[k]\left\{
\varepsilon_p L_p\left ( \frac{L_p}{3D_{ep}}+\frac{L_s}{2D_{es}} \right ) 
-\varepsilon_n L_n\left ( \frac{L_n}{3D_{en}}+\frac{L_s}{2D_{es}} \right )
\right\}.
\end{equation}
with $C[k]=\frac{(1-t^+)}{AF}I[k]$ and $M^0=\varepsilon_n L_n \ovc_{en}[0]+\varepsilon_s L_s \ovc_{es}[0]+\varepsilon_p L_p\ovc_{ep}[0]$ given by the initial conditions.

We deduce the average concentrations with (see \eqref{eq::electrolyte_conc})
\begin{equation}\label{discrete:electrolyte}
\ovc_{en}[k]=\ovc_{es}[k]+C[k]\left( \frac{L_n}{3D_{en}}+\frac{L_s}{2D_{es}}\right ),\qquad 
\ovc_{ep}[k]=\ovc_{es}[k]-C[k]\left( \frac{L_p}{3D_{ep}}+\frac{L_s}{2D_{es}}\right ).
\end{equation}
and the concentrations at the interfaces {(see \eqref{eq::electrolyte_bound})
\begin{align*} 
c_{en}[k](x_{cn})=\ovc_{en}[k] +C[k]\frac{L_n}{6D_{en}},&\quad 
c_{es}[k](x_{en})=\ovc_{es}[k] +C[k]\frac{L_s}{2D_{es}},\\
c_{es}[k](x_{ep})=\ovc_{es}[k] -C[k]\frac{L_s}{2D_{es}},&\quad
c_{ep}[k](x_{cp})=\ovc_{ep}[k] -C[k]\frac{L_p}{6D_{ep}}.
\end{align*}

\subsubsection{A major issue: the Diffusion-Limited Current Density}
During a charge or a discharge, and after a short transition, we have a steady-state situation for the ion \ce{Li^+} parameterized by time, {\it i.e.} there is no time derivative (see \eqref{eq::conservation_Lin_steady}, \eqref{eq::conservation_Lip_steady}, \eqref{eq::conservation_Lis_steady}). Therefore, the cation concentration is a decreasing function in space for a discharge and an increasing function in space for the charge; thus, the concentration reaches the extrema at $x_{cn}$ and $x_{cp}$, given by the values $\hatc_{en,c}(I)$ and $\hatc_{ep,c}(I)$, respectively (see \eqref{eq::constant_concentration_hat_defn}) as a function of intensity. 

From the physical point of view, negative concentration is not admissible; we conclude that the intensity has to range in an interval such that both concentrations are non-negative, the null case being the limit leading to the diffusion-limited current density \cite{NT04}. Two situations arise.
\begin{itemize}
\item Discharge case (positive intensity): The function  $\hatc_{en,c}(I)$ (see \eqref{eq::constant_concentration_hat}) is a decreasing function in relation to the intensity; therefore, there is an upper limit $I_{e}^+\geq 0$ such that $\hatc_{en,c}(I_{e}^+)=0$. Any intensity $I>I_{e}^+$ is not eligible; otherwise, we have a negative concentration at the cathode collector.
\item Charge case (negative intensity): The function  $\hatc_{ep,c}(I)$ (see \eqref{eq::constant_concentration_hat}) is an increasing function in relation to the intensity; therefore, there is a lower limit $I_{e}^-\leq 0$ such that $\hatc_{en,c}(I_{e}^-)=0$. Any intensity $I<I_{e}^-$ is not eligible; otherwise, we have a negative concentration at the anode collector. 
\end{itemize}
Consequently, during the charge or discharge stage,  ion diffusion is a limiting factor that yields the condition $I\in[I_{e}^-,I_{e}^+]$. We give an easy way to compute the approximation of the diffusion-limited current density.

As noted in remark \ref{rem_I-linearity-boundary},  $\hatc_{en,c}(I)$ and $\hatc_{en,c}(I)$ are affine functions of the intensity (see relations \eqref{eq::constant_concentration_hat}), enabling computation of the two bounds $I_{e}^-,I_{e}^+$. 
When $I=0$, the solution is the initial concentration $c_e^0$ and we have $\hatc_{en,c}(0)=\hatc_{ep,c}(0)=c_e^0$. Taking arbitrarily $I=1$, we compute the values $\hatc_{en,c}(1)$ and $\hatc_{ep,c}(1)$ by solving the system \eqref{eq::average_concentration_0D_hat} with $I=1$ and the relations \eqref{eq::constant_concentration_hat}. Taking advantage that the function is affine, we deduce that the left and right concentrations at the collectors go to zero for 
$$
I_e^-=\frac{c_e^0}{c_e^0-\hatc^1_{en,c}}<0,\qquad I_e^+=\frac{c_e^0}{c_e^0-\hatc^1_{ep,c}}>0,
$$ 
respectively.

\begin{defn}
$I_{e}^+$ and $I_{e}^-$ are, respectively, the highest and lowest diffusion-limited current densities. They represent the bounds to which the intensity $I$ must belong to guarantee the physical eligibility of the concentration. Their values are closely connected to the diffusion coefficients $D_{ek}, k \in \{n,s,p\}$ and assess the capacity to transport the cation from one electrode to another. 
Such intensities are a strong limitation on the time required for charging and discharging the cell since they represent a physical limit to the intensity. \qedsymbol
\end{defn}
For example, using the data of Table \ref{table::parameters}, we compute $I_{e}^+=7.0\,mA$ and $I_{e}^-=-7.7\,mA$.
}

\subsection{Numerical solution for the potential}\label{sec:numerical_potential_problem}
Assume that the average concentrations $\ovc_{n}[k]$, $\ovc_{p}[k]$, $\ovc_{en}[k]$, $\ovc_{es}[k]$, $\ovc_{ep}[k]$ together with the positive ionic concentrations at the positions $x_{cn}$, $x_{en}$, $x_{ep}$, $x_{cp}$ are given at time $t^k$. We want to determine the external intensity $I[k]$ and the other electric quantities by solving the steady-state problem under the constraint of the target intensity $I_0[k]=I_0(t^k)$ and the external electromotive potential $U_{ext}[k]=U_{ext}(t^k)$. 

The preliminary stage consists of computing $\iota_n[k]=\iota_n\big (\ovc_n[k],\ovc_{en}[k]\big )$, $\iota_p[k]=\iota_p\big (\ovc_p[k],\ovc_{ep}[k]\big )$, together with \eqref{eq::exact_DELTA_c},
$$
\Delta_c[k]=\frac{2t^+-1}{2f}
\ln \left ( \frac{c_{ep}[k](x_{cp})\,c_{es}[k](x_{ep})}{c_{es}[k](x_{en})\,c_{en}[k](x_{cn})}\right ).
$$
Moreover, the residual potential  is denoted by $V_{res}[k](I,R_{ext})$ as it depends on average concentrations through $\iota_n[k]$, $\iota_p[k]$ and $\Delta_c[k]$.

In the second stage, we compute $V_{00}[k]=V_{res}[k](0,0)$ which corresponds to the open-circuit configuration, given the external potential $U_{ext}(t^k)$. 
\begin{enumerate}
\item If $V_{00}[k]<0$ (charging), then by proposition \ref{prop:existence_uniqueness}, there exists a unique solution $I_{Max}[k]<0$ such that $V_{res}[k]( I_{Max}[k],0)=0$ we determine numerically. Two possibilities arise:
    \begin{enumerate}
    \item If $I_{Max}[k]\leq I_0[k]<0$, set the resistance $\displaystyle R_{ext}[k]=\frac{V_{res}[k](I_0[k],0)}{I_0[k]}$ and $I[k]=I_0[k]$.
    \item On the other hand, we have $I_0[k]<I_{Max}[k]<0$, the external resistance $R_{ext}[k]$ is set to zero and we cannot reach the target intensity. The intensity is then $I[k]=I_{Max}[k]$.
    \end{enumerate}
\item If  $V_{00}[k]>0$ (discharging), proposition \ref{prop:existence_uniqueness} implies that there is a unique $I_{Max}[k]>0$ such that $V_{res}[k](I_{Max}[k],0)=0$.
    \begin{enumerate}
    \item If  $0<I_0[k]<I_{Max}[k]$, we take $\displaystyle R_{ext}[k]=\frac{V_{res}[k](I_0[k],0)}{I_0}$ and set $I[k]=I_0[k]$.
    \item Otherwise, we cannot reach the target intensity and set $R_{ext}[k]=0$ and $I[k]=I_{Max}[k]$.
    \end{enumerate}
\end{enumerate}
We propose two numerical examples computed using the data in Table \ref{table::parameters}. 
\begin{enumerate}
\item For the discharge case, we assume that the external potential is null ($U_{ext}=0$).
\begin{itemize}
\item Considering the target intensity $I_0=50\,mA$, we reach it with $R_{ext}=64.6\,\Omega$ and the potential drop $\Phi_p-\Phi_n$ of the cell is $3.24\,V$.
\item Considering the target intensity $I_0=600\,mA$, we cannot reach it as the maximum $I_{Max }$ is $565\,mA$. The external resistance is $R_{ext}=0$ and the potential drop $\Phi_p-\Phi_n$ is almost equal to $0$ (short circuit).
\end{itemize}
\item For the charging case, we assume that the external potential is $U_{ext}=4.3\,V$.
\begin{itemize}
\item Considering the target intensity $I_0=-30\,mA$, we reach it with  external resistance $R_{ext}=3.57\,\Omega$ while the potential drop  $\Phi_p-\Phi_n$ is $4.19\, V$.
\item Considering the target intensity $I_0=-50\,mA$, we cannot reach it, as we have a maximum $I_{Max }$ equal to $-42.5\,mA$ and a null external resistance ($R_{ext}=0$). The potential drop  $\Phi_p-\Phi_n$ is precisely $4.3\,V$, equal to the assumed counter-electromotive potential $U_{ext}$.
\end{itemize}
\end{enumerate}

\subsection{Numerical solution for the complete model}\label{sec:numerical_potential_problem_complete_model}
Concentrations and potentials are now fully coupled through intensity $I$ as the leading variable. Now, assuming that all data have been computed at the time $t^{k-1}$, we seek the solution at the time $t^k$.

The problem is not a simple combination of the two previous problems, since several issues may arise, namely negative concentrations if the intensity is too high. We gather the whole implicit problem into two subgroups.
\begin{eqnarray*}
&& \textrm{\bf Group (1): the concentrations}\\
\ovc_{n}[k]&=&\ovc_{n}[k-1]-\Delta t\, \frac{I[k]}{A\,F\varepsilon_n \,L_n},\\
\ovc_{p}[k]&=&\ovc_{p}[k-1]+\Delta t\, \frac{I[k]}{A\,F\varepsilon_p \,L_p},\\
(\varepsilon_n L_n +\varepsilon_s L_s +\varepsilon_p L_p)\,\ovc_{es}[k]&=&M^0+\frac{(1-t^+)}{AF}I[k]\left [ 
\varepsilon_p L_p\left ( \frac{L_p}{3D_{ep}}+\frac{L_s}{2D_{es}} \right ) 
-\varepsilon_n L_n\left ( \frac{L_n}{3D_{en}}+\frac{L_s}{2D_{es}} \right )
\right ],\\
\ovc_{en}[k]&=&\ovc_{es}[k]+\frac{(1-t^+)}{AF}I[k]\left( \frac{L_n}{3D_{en}}+\frac{L_s}{2D_{es}}\right ),\\
\ovc_{ep}[k]&=&\ovc_{es}[k]-\frac{(1-t^+)}{AF}I[k]\left( \frac{L_p}{3D_{ep}}+\frac{L_s}{2D_{es}}\right ),\\
c_{en}[k](x_{cn})&=&\ovc_{en}[k] +\frac{(1-t^+)L_n}{6\,A\,F\,D_{en}}I[k],\\
c_{es}[k](x_{en})&=&\ovc_{es}[k] +\frac{(1-t^+)L_s}{2\,A\,F\,D_{es}}I[k],\\
c_{es}[k](x_{ep})&=&\ovc_{es}[k] -\frac{(1-t^+)L_s}{2\,A\,F\,D_{es}}I[k],\\
c_{ep}[k](x_{cp})&=&\ovc_{ep}[k] -\frac{(1-t^+)L_p}{6\,A\,F\,D_{ep}}I[k],
\end{eqnarray*}
\begin{eqnarray*}
&& \textrm{\bf Group (2): the potentials and intensity}\\
\Delta_c[k]&=&\frac{2t^+-1}{2f}
\ln \left ( \frac{c_{ep}[k](x_{cp})\,c_{es}[k](x_{ep})}{c_{es}[k](x_{en})\,c_{en}[k](x_{cn})}\right ),\\
\iota_n[k]&=&2F\,A\, L_n\,a_n\,r_n\, \Big (\ovc_{en}[k]\Big )^{\alpha_n} \Big (\ovc_n[k]\,(c^{max}_n-\ovc_n[k])\Big )^{1-\alpha_n},\\ 
\iota_p[k]&=&2F\,A\, L_p\,a_p\,r_p\, \Big (\ovc_{ep}[k]\Big )^{\alpha_p} \Big (\ovc_p[k]\,(c^{max}_p-\ovc_p[k])\Big )^{1-\alpha_p},\\
I[k]&=&+\iota_n[k]\ \bsinh\big (f\eta_{en}[k]/2\big ),\\
I[k]&=&-\iota_p[k]\ \bsinh\big (f\eta_{ep}[k]/2\big ), \\
V_{res}[k]&=&U_{ep}-U_{en}-U_{ext}[k]+\eta_{ep}[k]-\eta_{en}[k]-\Delta_{c}[k]-(R_{ext}[k]+R_{int})I[k], 
\end{eqnarray*}
where $R_{ext}[k]$ is a free parameter of the system. The objective is to determine an eligible intensity $I[k]$ as close as possible to the target intensity $I_0[k]$ such that $V_{res}[k]=0$.

\subsubsection{Preliminaries}
In the first step, we determine whether we have a configuration for charging or discharging the cell at time $t^{k}$ as a function of the external potential $U_{ext}[k]$ applied to the circuit. Following the methodology proposed at the beginning of subsection \ref{sec:numerical_potential_problem}, we test the model with $I=0$, and from the equations of Group (1), we immediately deduce that $\ovc_{n}[k]=\ovc_{n}[k-1]$, $\ovc_{p}[k]=\ovc_{p}[k-1]$, while the ion concentration is constant in space. On the other hand, from Group (2), we have $\Delta_c=0$ and $\eta_{en}[k]=0$, $\eta_{ep}[k]=0$. Finally, we compute the residual potential $V_{00}[k]$ and now the rule is
\begin{itemize}
\item If $V_{00}[k]<0$, we are dealing with a charge configuration and $I[k]$ is non-positive,
\item If $V_{00}[k]>0$, we are dealing with a discharge configuration, and $I[k]$ is non-negative.
\end{itemize}
However, the diffusion-limited current densities are independent of time step and depend only on the initial concentration of \ce{Li^+} and the diffusion coefficients.  We then have the constraints
\begin{itemize}
\item charge case: eligible intensity has to satisfy $I[k]\in [I_e^-,0]$;
\item discharge case: we must have $I[k]\in [0,I_e^+]$.
\end{itemize}

\subsubsection{Case $I_0[k]\in ]I_e^-,I_e^+[$} If $I_0[k]$ is eligible we compute Group (1) with $I_0[k]$ and then compute $V_0[k]=V_{res}[k](I_0[k],0)$.
\begin{itemize}
\item If $V_0[k]$ and $V_{00}[k]$ have the same sign, then we reach the target intensity $I[k]=I_0[k]$ and define the external resistance $\displaystyle R_{ext}[k]=\frac{V_0[k]}{I_0[k]}$;
\item Otherwise, $V_0[k]$ and $V_{00}[k]$ have opposite signs, and we cannot reach the target values. The proposition \ref{prop:existence_uniqueness} states that there exists a unique intensity $I_{Max}[k]$ such that $V_{res}[k](I_{Max}[k],0)=0$. We then take $I[k]=I_{Max}[k]$ and $R_{ext}[k]=0$. Note that $I_{Max}[k]$ is also eligible since $I_0[k]$ is assumed admissible.
\end{itemize}
\begin{remark}
Note that taking $I[k]=I_e^+$ or $I[k]=I_e^-$ is impossible, since we get a null concentration. Consequently, $\Delta_c[k]$ is not defined (the logarithmic function $\ln$ requires the argument to belong to $]0,+\infty[$). \qedsymbol
\end{remark}
\subsubsection{Case $I_0[k]\notin ]I_e^-,I_e^+[$}
This is a difficult situation because the target intensity $I_0[k]$ is too high and causes negative concentrations. We have to find a lower intensity $I[k]\in [I_e^-,I_e^+]$ and an external resistance $R_{ext}[k]$ that preserve the electrical balance, {\it i.e.} $V_{res}[k]=0$, together with non-negative concentrations. 

We first treat the discharge case, that is, we have $V_{00}[k]>0$, $I_0[k]\geq I_e^+>0$. To achieve an approximation of $I[k]$ and $R_{ext}[k]$, we propose a dichotomy-like method that provides a bounded interval of the intensity. We detail hereafter the method and skip index $k$ for the sake of readability. We create a sequence of intervals $\big [I_a^\ell,I_b^\ell\big ]$ that we initialize with $I_a^0=0$ and $I_b^0=I_e^+$. The algorithm is as follows:
\begin{enumerate}
\item Take $\displaystyle I_c^\ell=\frac{I_a^\ell+I_b^\ell}{2}$.
\item Compute the associated concentrations (they are all strictly positive since $I_c^\ell<I_e^+$).
\item Compute $\Delta^\ell_c$, $\iota_n^\ell$, $\iota_p^\ell$ and then deduce $\eta_{en}^\ell$, $\eta_{ep}^\ell$.
\item Compute $V_{0}^\ell=V_{res}^\ell(I_c^\ell,0)$. 
\item If $V_{0}^\ell>0$ (the intensity is too low) we take $I_a^{\ell+1}=I_c^{\ell}$ and $I_b^{\ell+1}=I_b^\ell$, otherwise (intensity is too high) we take $I_b^{\ell+1}=I_c^{\ell}$ and $I_a^{\ell+1}=I_a^\ell$. 
\item Repeat until $|V_{res}^\ell|\leq \texttt{tol}$ and $|I_a^{\ell+1}-I_b^{\ell+1}|\leq \texttt{tol}$ where \texttt{tol} is a given tolerance.
\end{enumerate}
At the end of the algorithm, we take the intensity $I[k]=I_c^\ell<I_e^+$ that guarantees all positive concentrations and $V_{res}(I_c^\ell,0)\approx 0$. Consequently, we take $R_{ext}[k]=0$ and $I[k]=I_c^\ell$.

For the charge case, the algorithm is very similar by taking $I_a^0=I_e^-$ and $I_b^0=0$. We reach an approximation $I_e^-<I[k]=I_c^\ell<0$ such that $V_{res}(I_c^\ell,0)\approx 0$ and $R_{ext}[k]=0$.

\section{Simulations}
A battery composed of an anode and a cathode with constant open-circuit potential (no intercalation versus plating is considered in the example) separated by a porous medium is analysed with the numerical simulator. Data are given in Table \ref{table::parameters}. 
\begin{table}[ht]
\small
\renewcommand{\arraystretch}{1.4}
\begin{tabular}{| l |l| l |}
\hline
parameter  & dimension& name\\
\hline\hline
$A=10^{-4}$                                      &$[m^2]$        & cross area of the cell   \\
$L_n=80.\,10^{-6}$, $L_s=25.\,10^{-6}$, $L_p=100.\,10^{-6}$&$[m]$          & length of the domains\\
$a_{n}=885000$,\hskip 1em $a_{p}=723000$       &$[m^{-1}]$     &  specific interfacial area of the electrodes\\
$\varepsilon_n=0.36$,\hskip 1em  $\varepsilon_s=0.30$, \hskip 1em $\varepsilon_p=0.36$        &$[m^3\,m^{-3}]$&  porosity\\
\hline
$F=96487$                                      &$[C\,mol^{-1}]$&  Faraday constant \\
$f=\frac{F}{RT}=38.68$                         &$[C\, J^{-1}]$ &  inverse thermal potential   \\
$t^+=0.364$                                    &$[-]$          &  transfer number\\
$\sigma_{n}=100$, \hskip 1em $\sigma_{p}=0.5$   &$[S\,m^{-1}]$  & effective electronic conductivity in electrode\\
$\alpha_n=0.5$,\hskip 1em  $\alpha_p=0.5$      &$[-]$          & electrodes transfer coefficients \\
$D_{en}=1.65\,10^{-11}$, $D_{es}=2.0\,10^{-10}$, $D_{ep}=4.15\,10^{-11}$   &$[m^2 s^{-1}]$& cation diffusivity \\
$\sigma_{en}=0.1$,\hskip 1em  $\sigma_{es}=0.1$,\hskip 1em $\sigma_{ep}=0.1$                     &$[S\,m^{-1}]$ & effective ionic conductivity \\
\hline
$c_{n}^{max}=50000$,\hskip 1em  $c_{p}^{max}=30000$  &$[mol\,m^{-3}]$& maximum mole concentration in the electrodes\\
$r_n=2.7\,10^{-11}$,\hskip 1em $r_p= 1.8\,10^{-12}$  &$[m^{5/2} mol^{-1/2} s^{-1}]$ & electrochemical reaction rate constant\\
$c_{en}^{0*}=10^3$,\hskip 1em  $c_{ep}^{0*}=10^3$         &$[mol\,m^{-3}]$& reference cation mole concentration \\
$c_{n}^{*}=25000$, \hskip 1em $c_{p}^{*}=15000$      &$[mol\,m^{-3}]$& reference reduced material\\
$i_{en}^{0*}=2.06$,\hskip 1em  $i_{ep}^{0*}=0.082$       &$[A\,m^{-2}]$  & reference current surface density \\ 
$j_{en}^{0*}=1.82\,10^6$,\hskip 1em  $j_{ep}^{0*}=5.96\,10^{4}$       &$[A\,m^{-3}]$  & reference current volume density \\ 
$U_{en}=0.0$,\hskip 1em  $U_{ep}=3.8$        &$[V]$          & open circuit potential of anode and cathode\\
\hline
\end{tabular}
\caption{Parameters for the benchmarking}\label{table::parameters}
\end{table}

We consider three different situations by modifying the dielectric diffusion coefficients.
\begin{itemize}
\item Experiment ONE is the reference case using the dielectric diffusion coefficients given in Table \ref{table::parameters}. 
\item Experiment HALF consists of using the same diffusion coefficients but divided by two (low diffusion case). 
\item Experiment THREE uses three times the values of the diffusion coefficients (high diffusion case). 
\end{itemize}
The purpose of the simulation is to show the impact of dielectric diffusion in the charging and discharging processes. More specifically, we highlight the current limitation due to diffusion of the cations in the cell.

In Figure \ref{fig:external_condition}, we plot the target intensity together with the external potential that we would like to achieve during the 3600 s simulation. Note that the $1C$ rate configuration of a total charge of $4.3418$ Coulomb corresponds to the constant intensity $I=1.2\,mA$, required to transfer the whole charge in one hour. The scenario is as follows ($I_0$ being the target intensity):
\begin{itemize}
\item $t\in [0,900]$, discharge $5C$ rate, $(I_0=6\,mA)$ and $U_{ext}=0$, 
\item $t\in[900,1200]$, charge  $5C$ rate, $(I_0=-6\,mA )$ and $U_{ext}=4.2$,
\item $t\in [1200,1500]$, open circuit and cell at rest, 
\item $t\in [1500,1800]$, discharge  $15C$ rate, $(I_0=18mA)$ and $U_{ext}=0$,
\item $t\in[1800,2400]$, charge $15C$ rate, $(I_0=-18mA)$ and $U_{ext}=4.5$,
\item $t\in [2400,3600]$, discharge with $5C$ rate, $(I_0=6\,mA)$ and $U_{ext}=0$,
\end{itemize}
\begin{figure}[ht]
    \centering
    \includegraphics[width=0.75\linewidth]{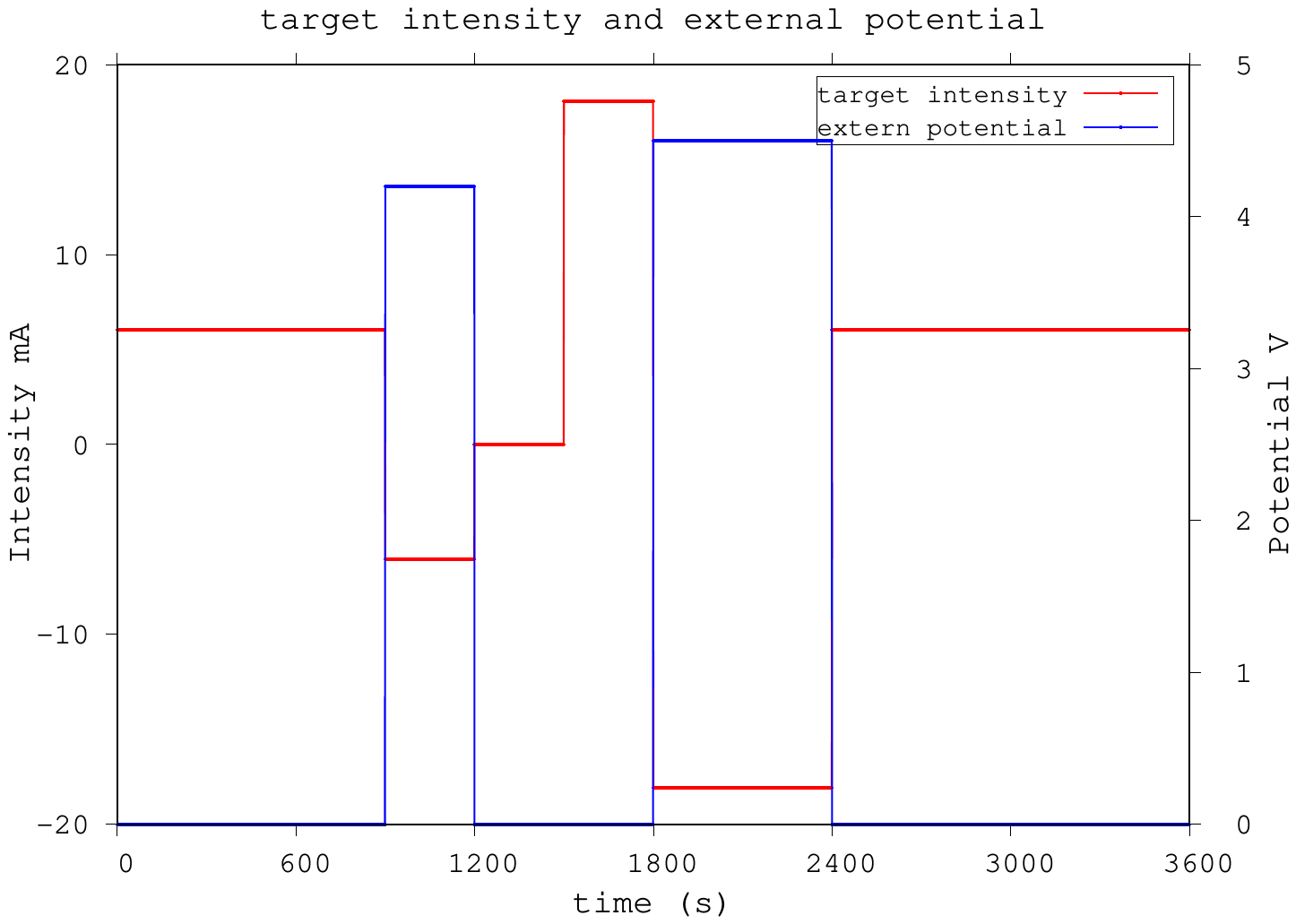}
    \caption{Target intensity and external potential for the numerical simulation as a function of time}
    \label{fig:external_condition}
\end{figure}

In Figure \ref{fig:Intensity_voltage}, we plot the intensity across the cell together with the $I_e^+$ (red line) and $I_e^-$ (blue line) during the simulation for the three scenarios. We observe that the intensity strictly remains between the two constraints. For high diffusion coefficients (left panel), the cell intensity is not subject to the non-negative concentration constraints. Moreover, in some situations,  we do not reach the target intensity, as shown by the zero external resistance in Figure \ref{fig:potential_and_external_resistance} (left panel).

A lower dielectric diffusion coefficient makes the intensity eligibility more restrictive, and we observe in scenario ONE that a 15C intensity is not possible due to the concentration constraint resulting from the dielectric diffusion coefficients. The scenario HALF confirms the deep impact of the concentrations' positivity (and consequently, of the diffusion coefficients). Most of the time, the intensity is bounded by the diffusion-limited intensity, and the cell is not able to deliver a high current.   

\begin{figure}[ht]
    \centering
    \includegraphics[width=0.33\linewidth]{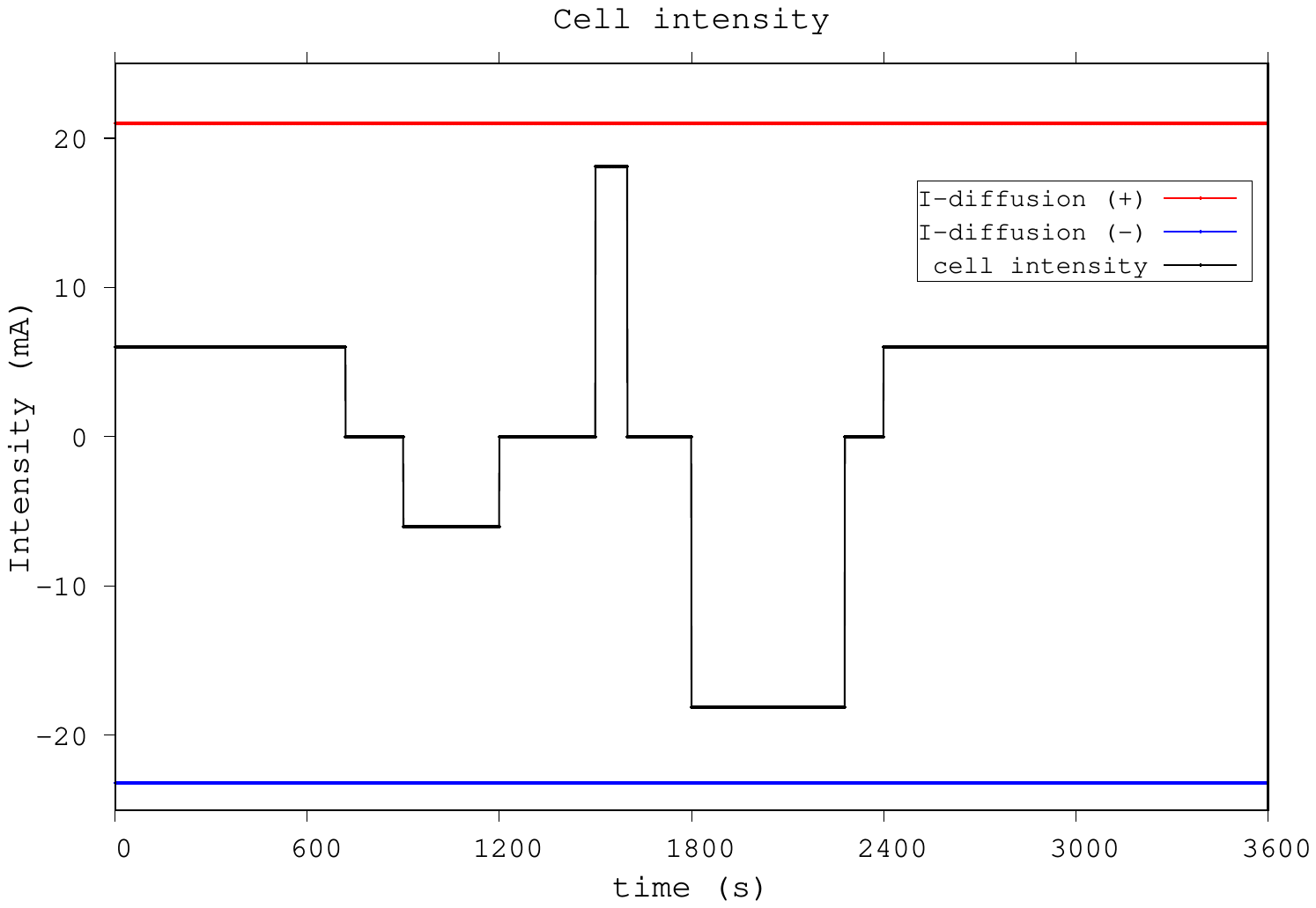}
    \includegraphics[width=0.33\linewidth]{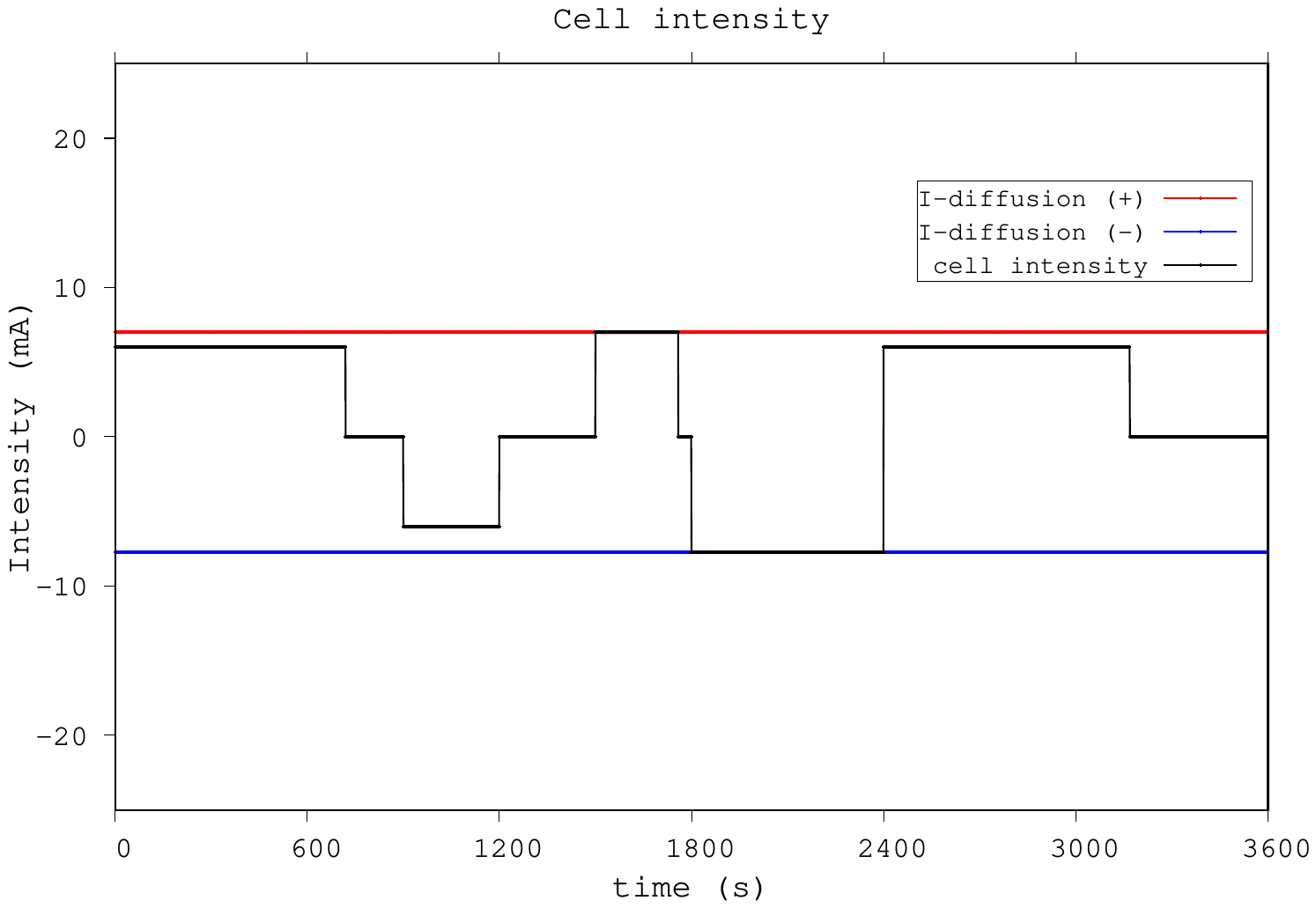}
    \includegraphics[width=0.33\linewidth]{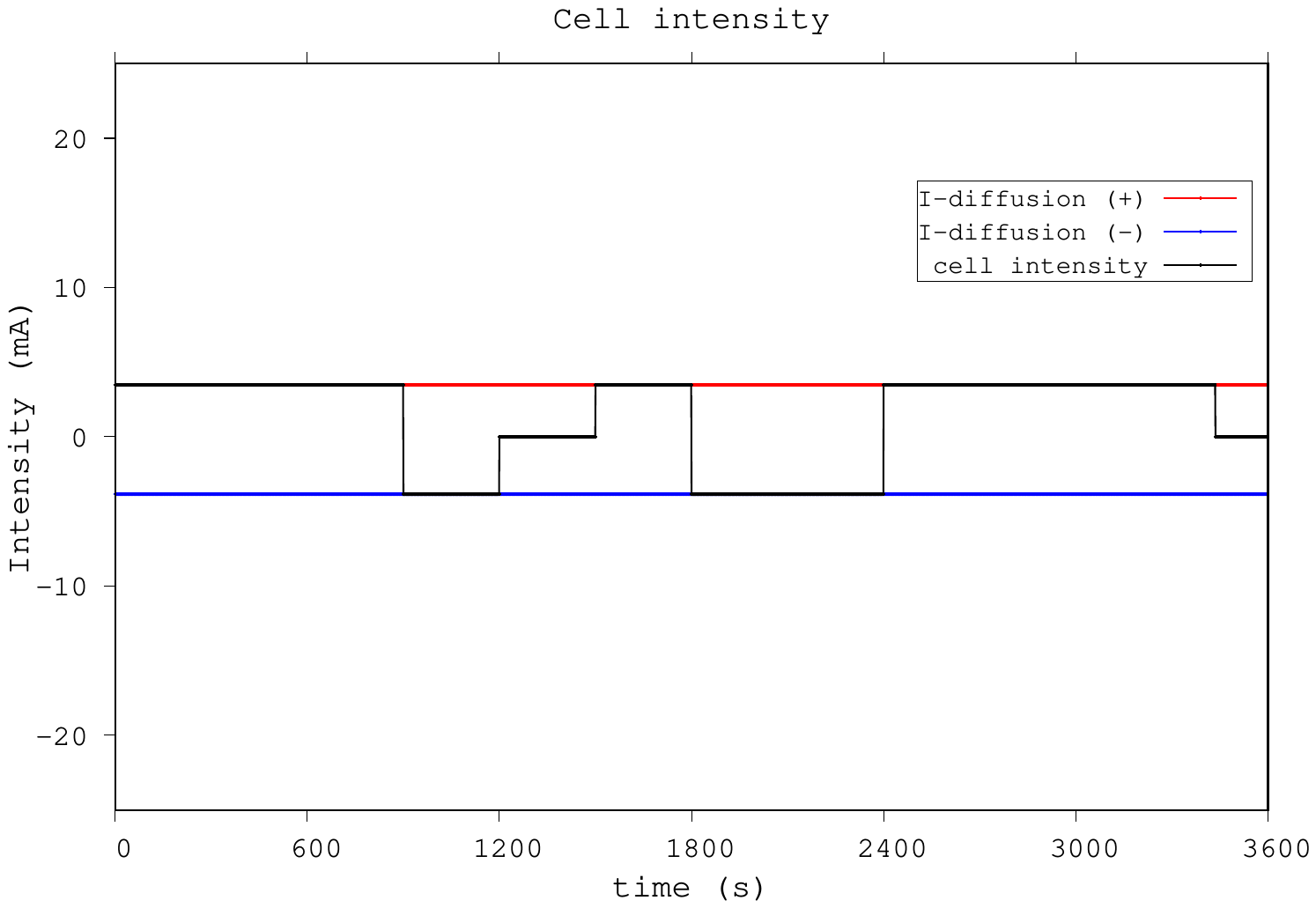}
    \caption{\captionpolice Cell intensity together with the upper and lower diffusion-limited intensity. Scenario THREE (left panel), scenario ONE (middle panel), and scenario HALF (right panel)}
    \label{fig:Intensity_voltage}
\end{figure}

As a direct consequence of the intensity, the mass transfer between the electrodes is strongly limited when dealing with a low diffusion coefficient, as shown in Figure \ref{fig:electrode_concentration}. We managed to reach full depletion of the anode for scenarios ONE and THREE with a 5C intensity, when $t \in [900, 1200]$, but the scenario HALF limits the maximum intensity to around a 4C discharge and keeps the cell from transferring the whole charge in less than 15 minutes.

\begin{figure}[ht]
    \centering
    \includegraphics[width=0.33\linewidth]{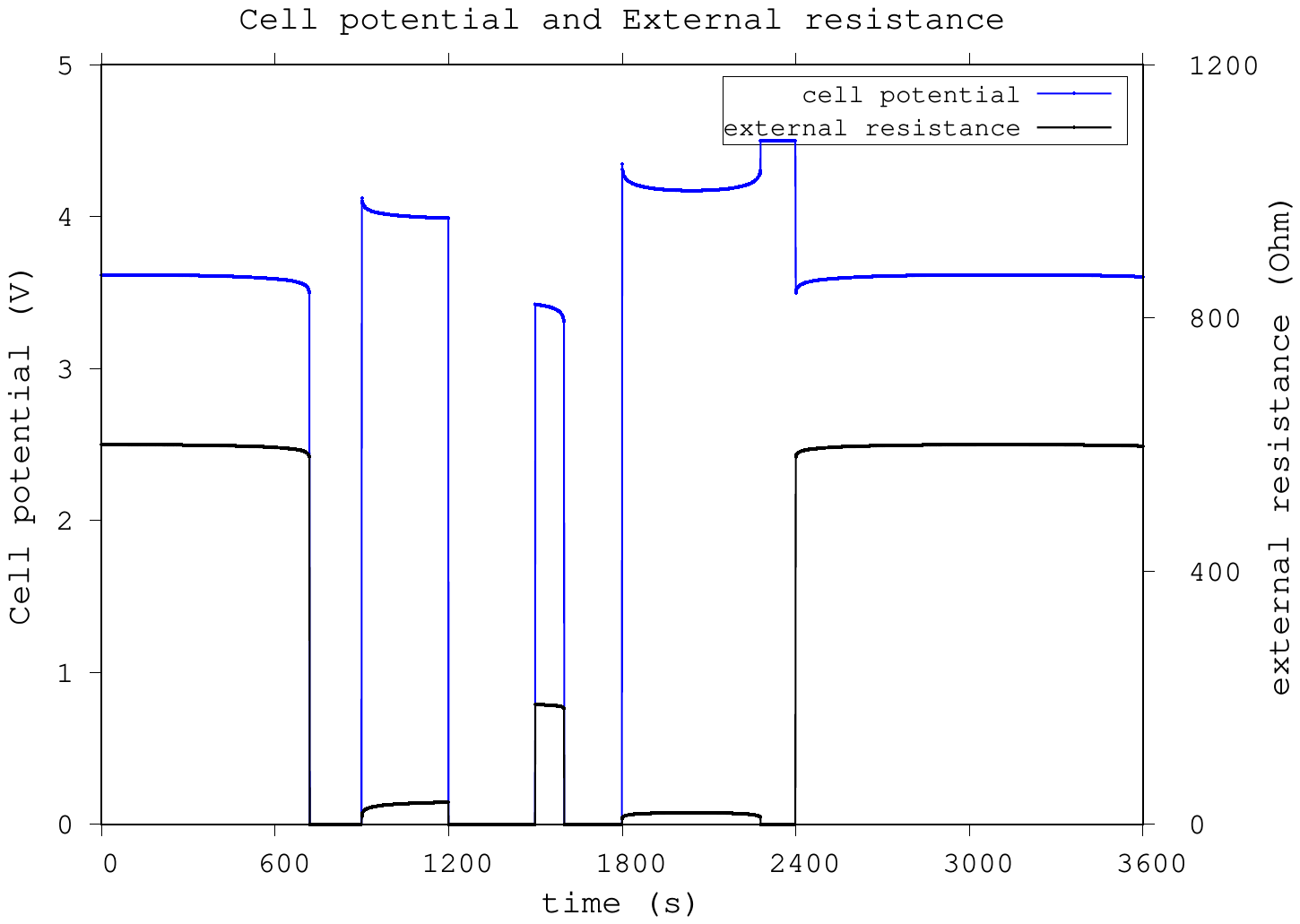}
    \includegraphics[width=0.33\linewidth]{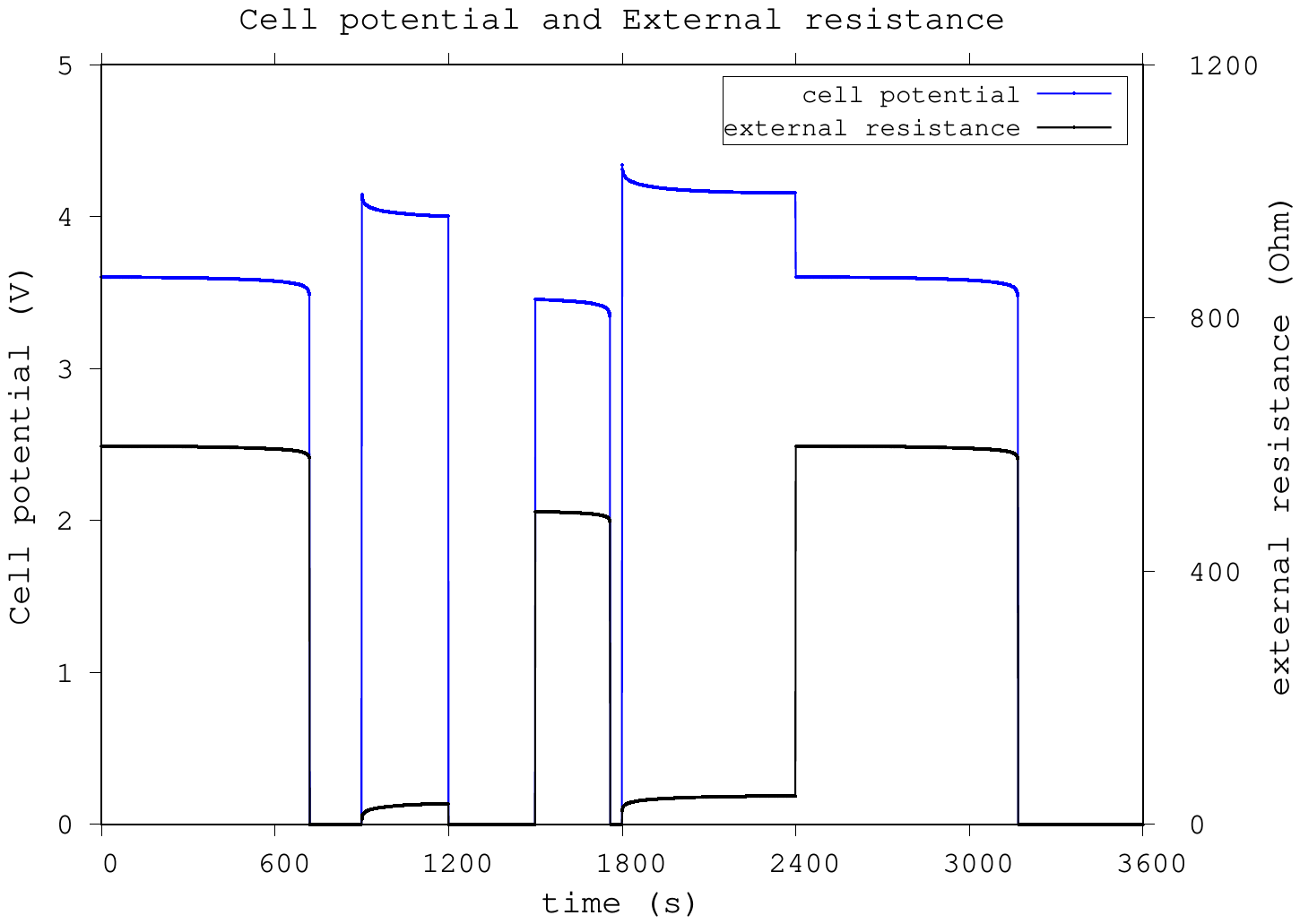}
    \includegraphics[width=0.33\linewidth]{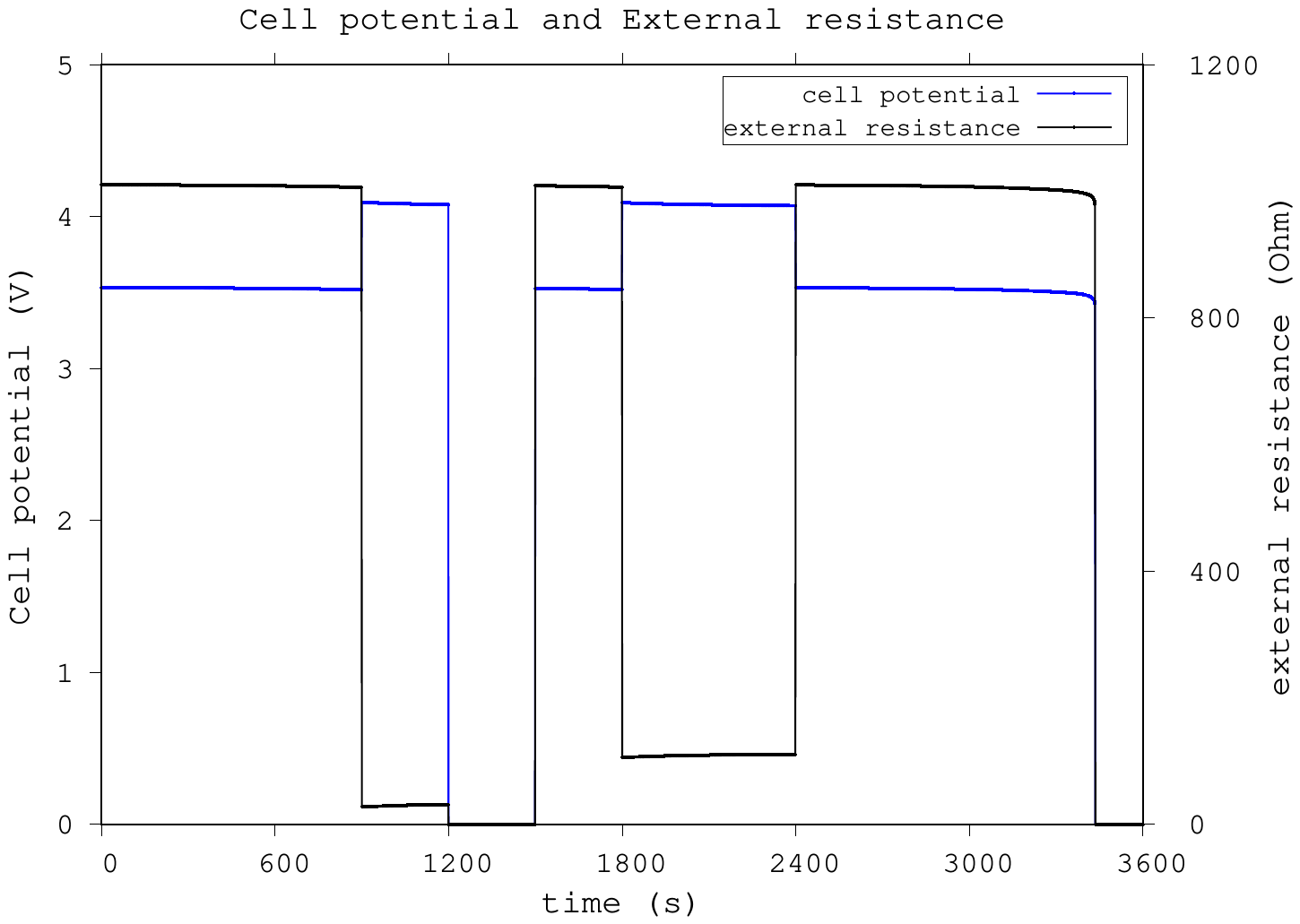}
    \caption{Cell potential and the associated external resistance. Scenario THREE (left panel), scenario ONE (middle panel) and scenario HALF (right panel) }
    \label{fig:potential_and_external_resistance}
\end{figure}

To highlight the positive concentration issue, we plot the concentration at the collector interfaces in Figure \ref{fig:collector_concentration}. We recall that for $I=I_e^-$, the concentration in the anode collector is almost null (charging case), whereas the concentration in the cathode collector is close to zero when $I=I_e^+$ (discharging case). For high diffusion coefficients, positivity does not lead to a limitation of the current (left panel), but a 15C charge, when $t \in [1800, 2400]$, is not achievable for scenario ONE (middle panel). Similarly, the current corresponding to the 5C charge, when $t \in [900, 1200]$, is impossible to achieve with the diffusion coefficients used in scenario HALF (right panel).  

\begin{figure}[ht]
    \centering
    \includegraphics[width=0.33\linewidth]{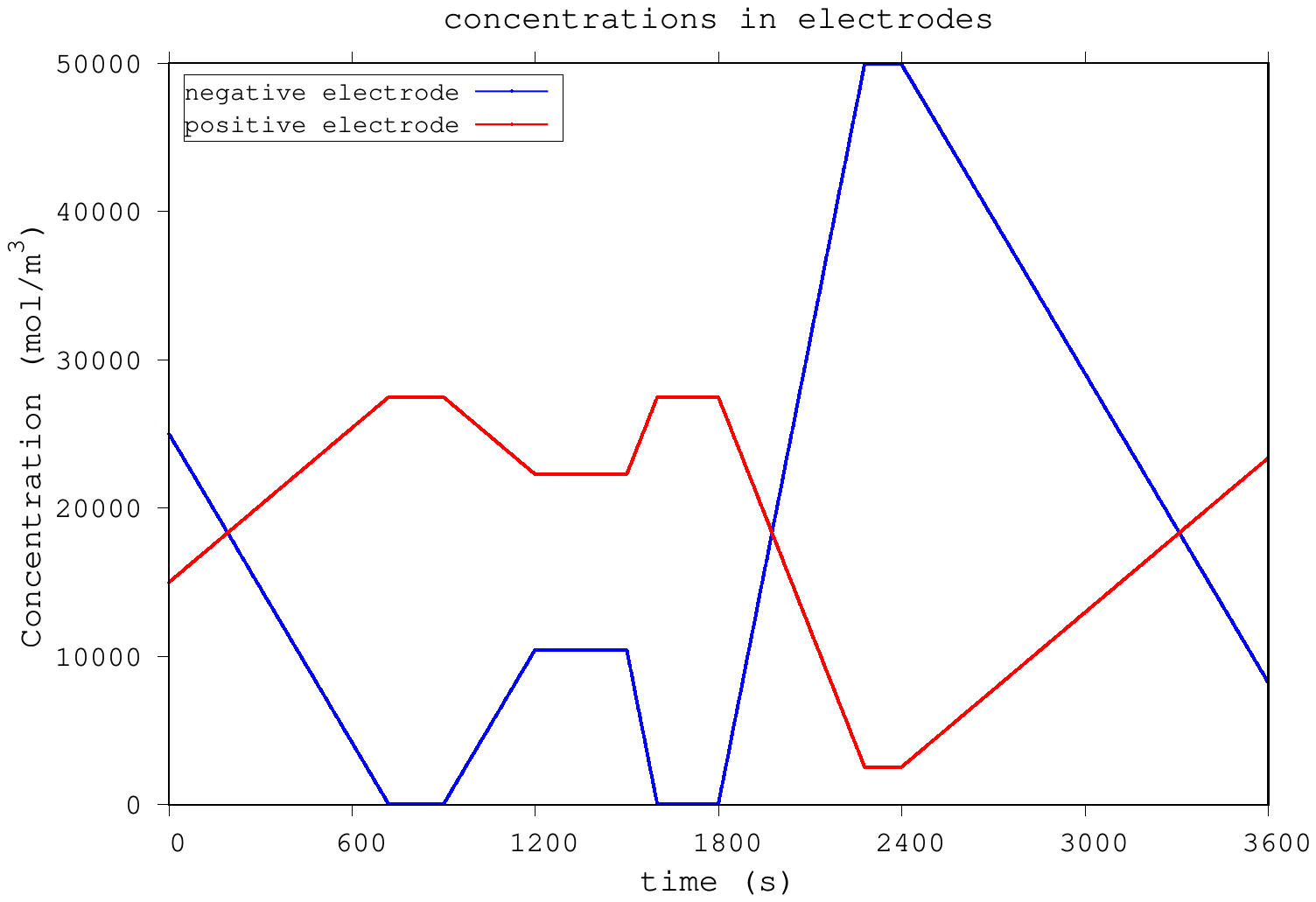}
    \includegraphics[width=0.33\linewidth]{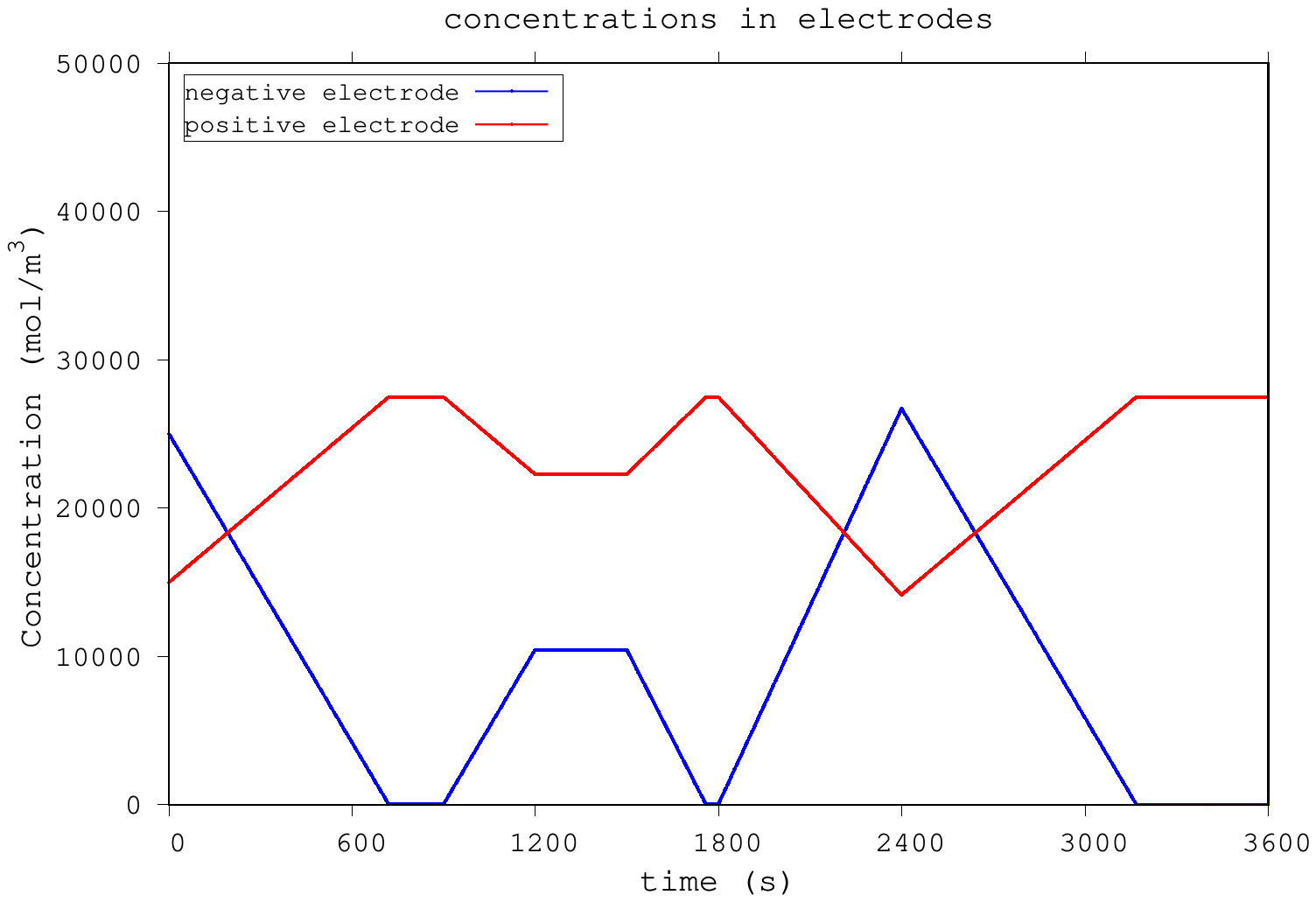}
    \includegraphics[width=0.33\linewidth]{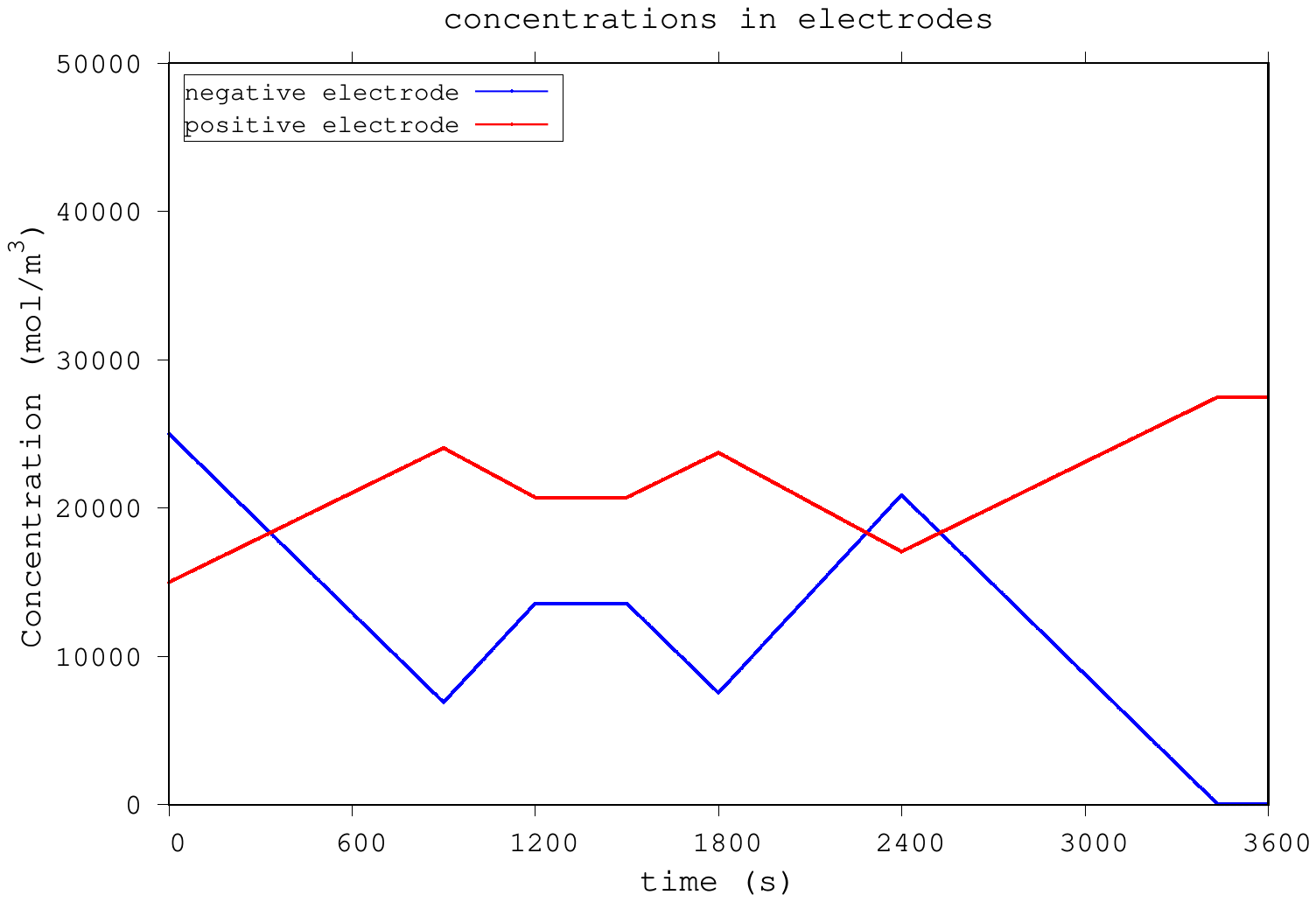}
    \caption{Concentration in the electrodes. Scenario THREE (left panel), scenario ONE (middle panel) and scenario HALF (right panel)}
    \label{fig:electrode_concentration}
\end{figure}

\begin{figure}[ht]
    \centering
    \includegraphics[width=0.33\linewidth]{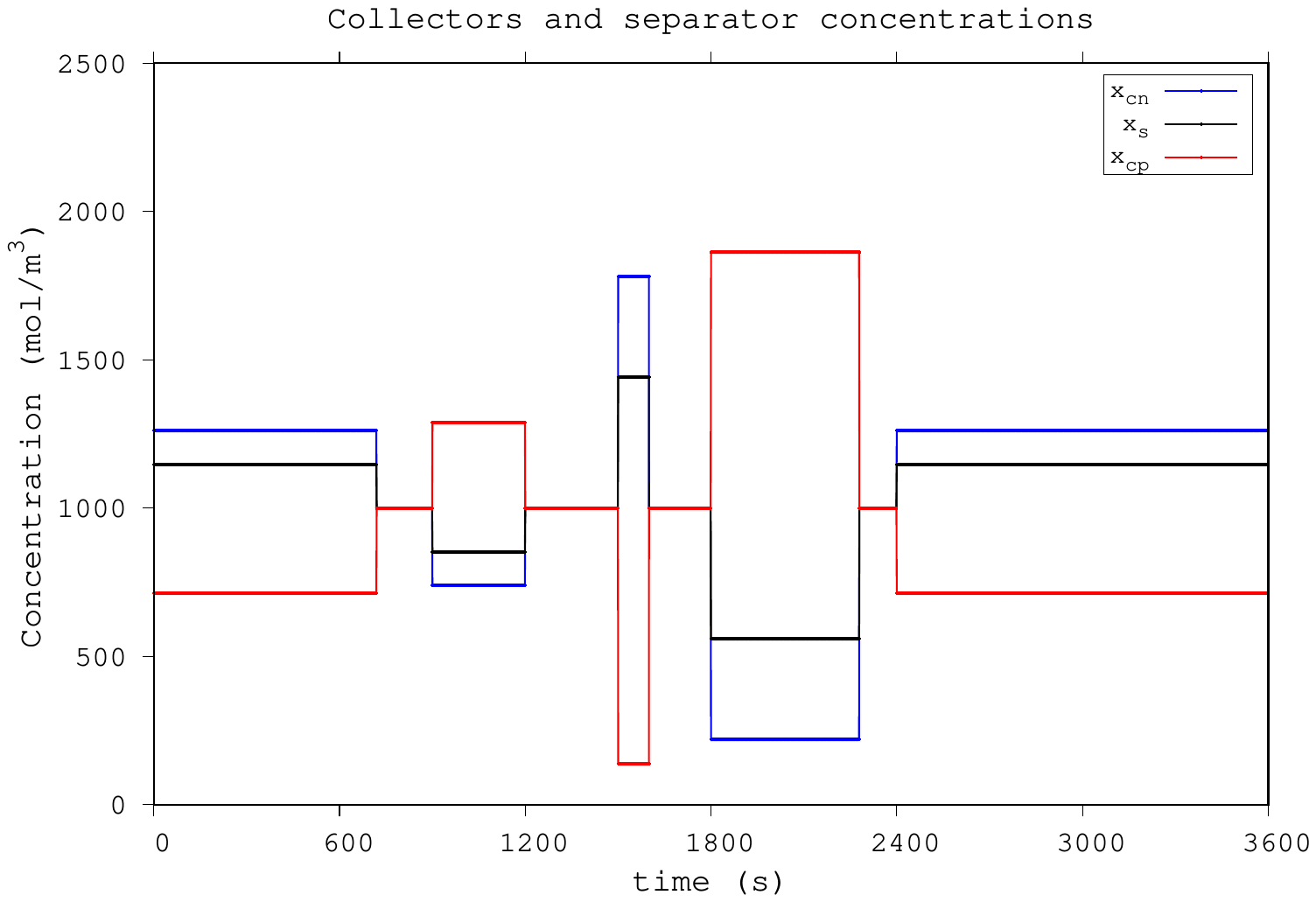}
    \includegraphics[width=0.33\linewidth]{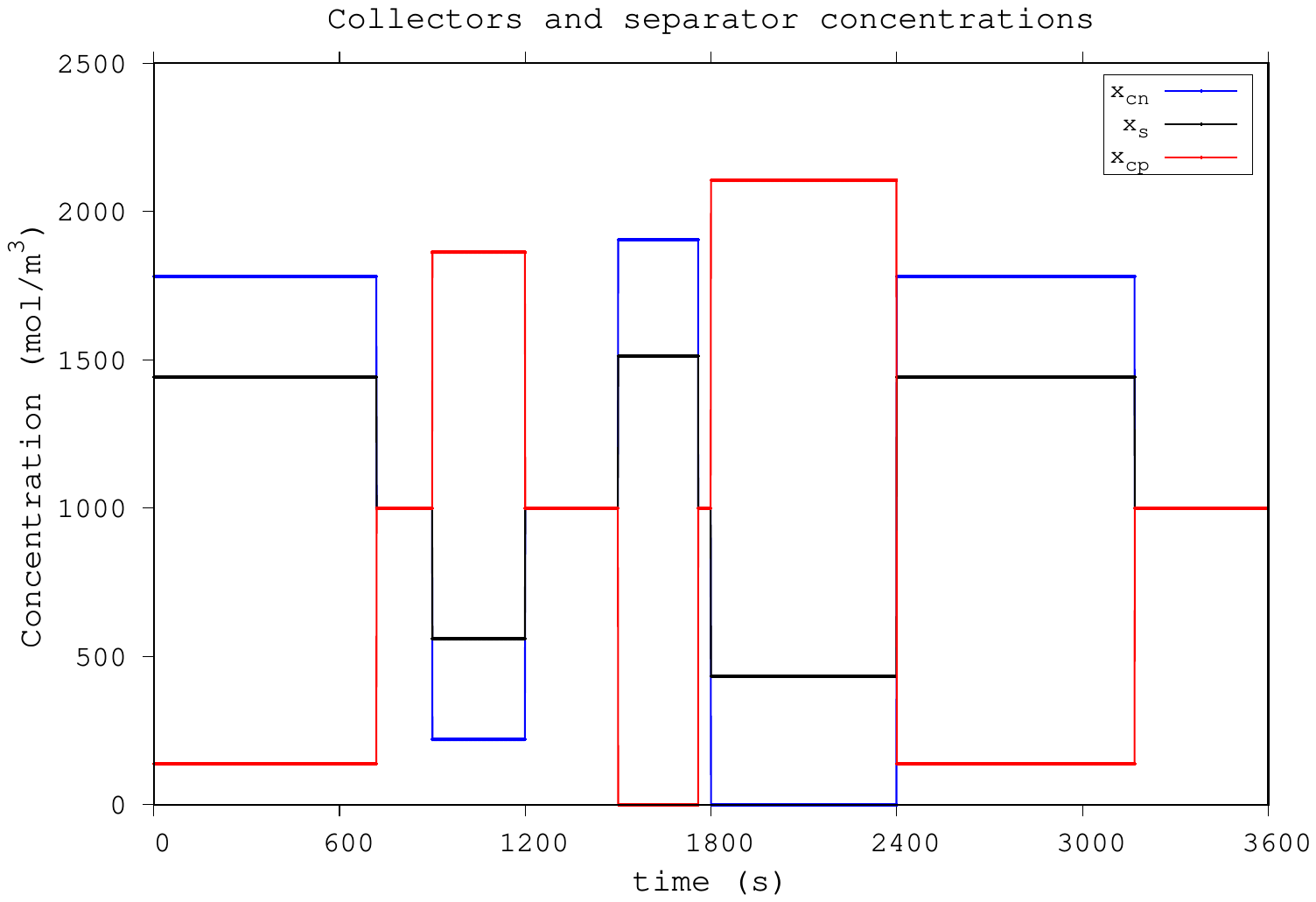}
    \includegraphics[width=0.33\linewidth]{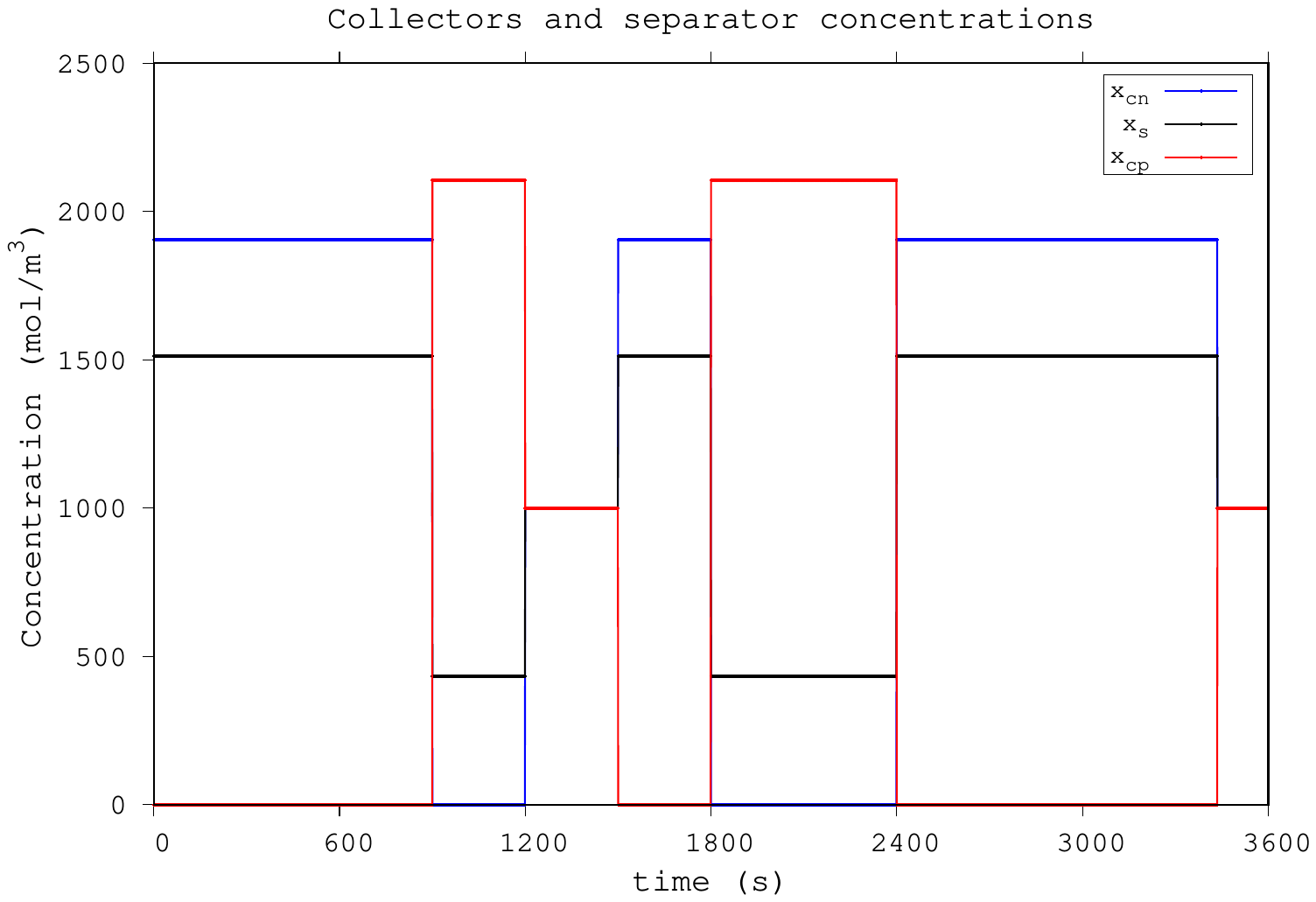}
    \caption{Electrolyte concentration at the collectors. Scenario THREE (left panel), scenario ONE (middle panel), and scenario HALF (right panel)}
    \label{fig:collector_concentration}
\end{figure}

Finally, in Figure \ref{fig:overpotential} we report the overpotentials and $\Delta_c$ for the three scenarios. We notice that the potential $\Delta_c$ has a strong impact on the HALF scenario since very low concentrations occur at the interfaces, which leads to a significant amount of potential regarding the potential difference of the whole cell. In contrast, a high diffusion coefficient makes $\Delta_c$ less than 1$\permil$.
\begin{figure}[ht]
    \centering
    \includegraphics[width=0.33\linewidth]{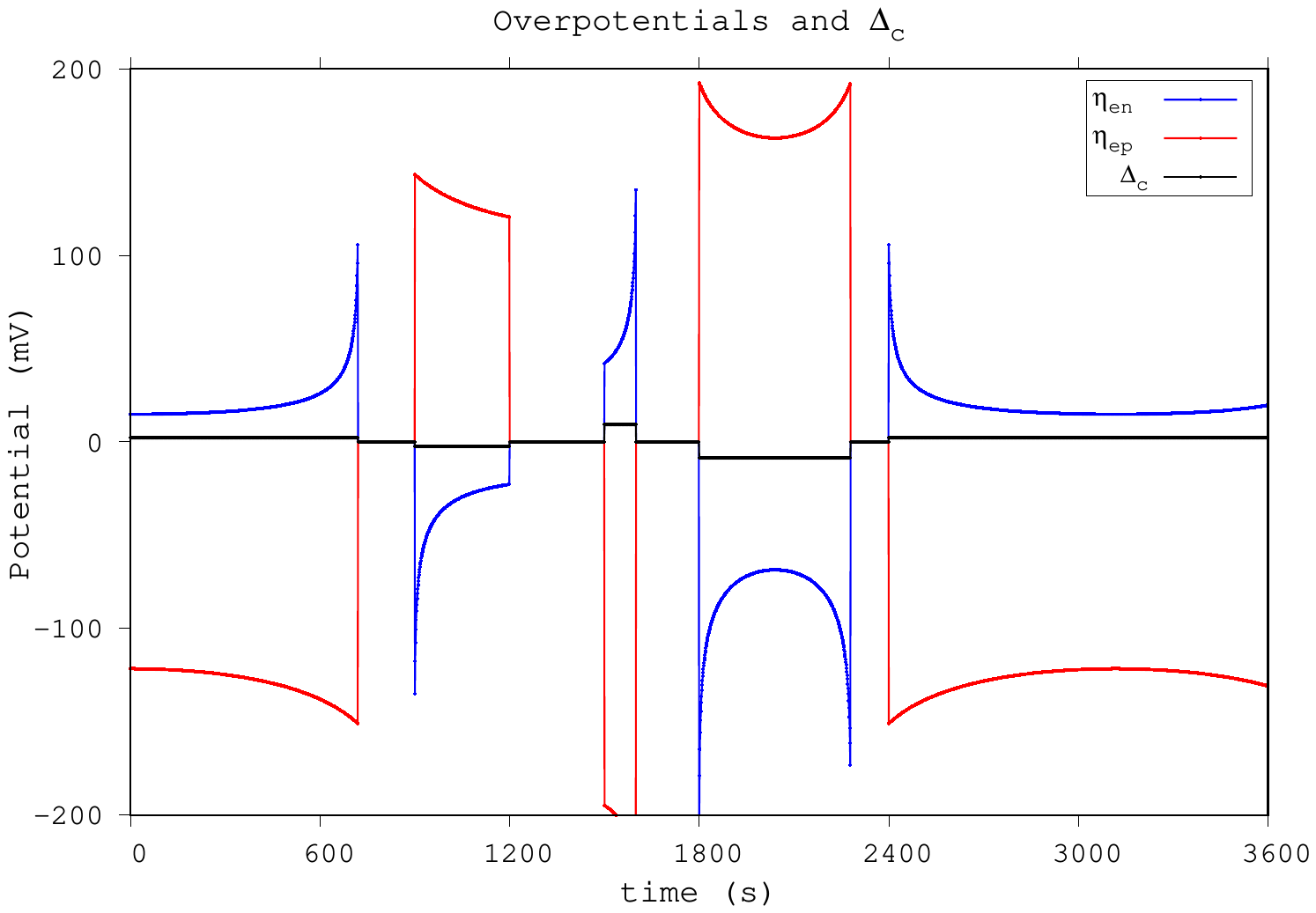}
    \includegraphics[width=0.33\linewidth]{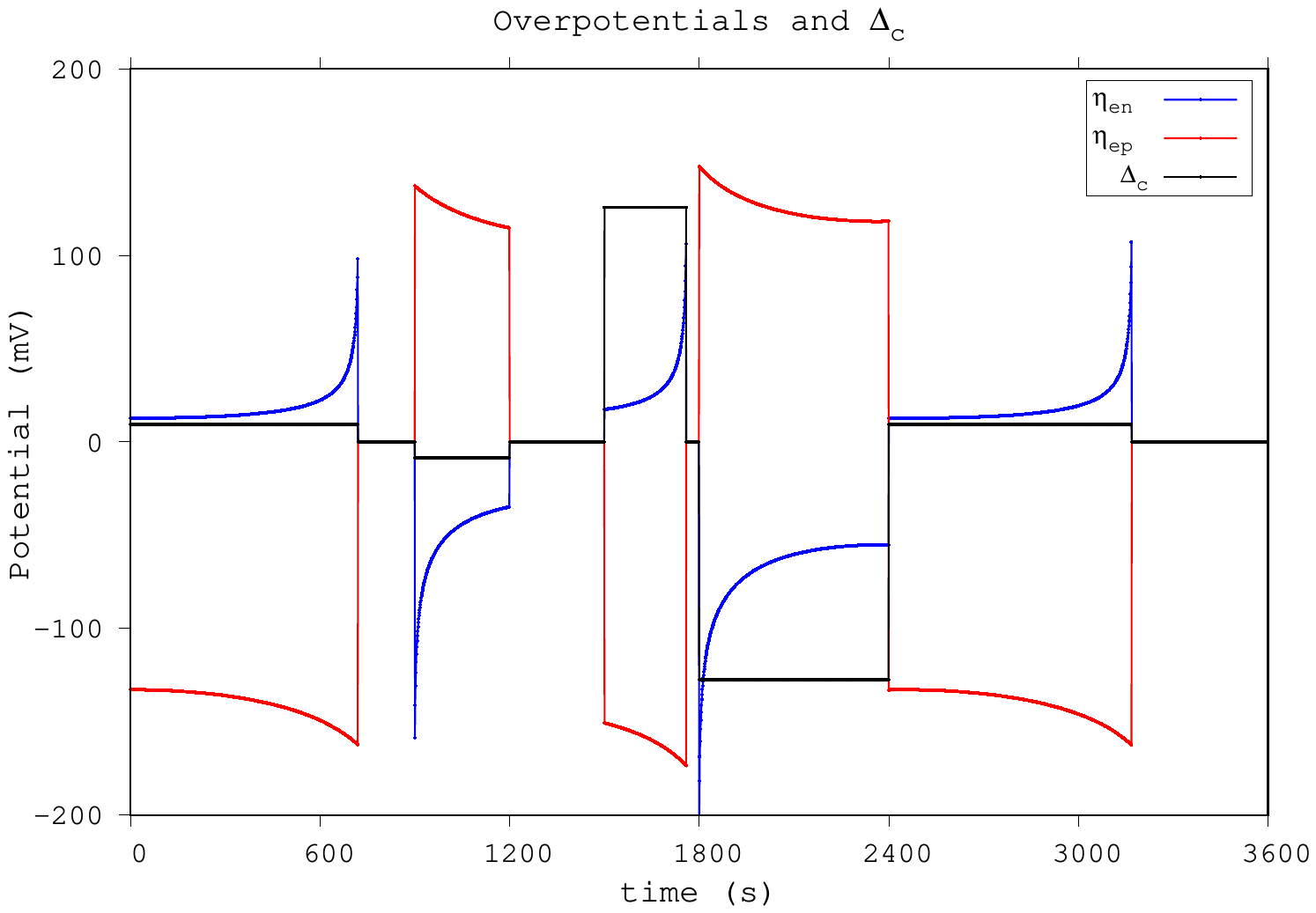}
    \includegraphics[width=0.33\linewidth]{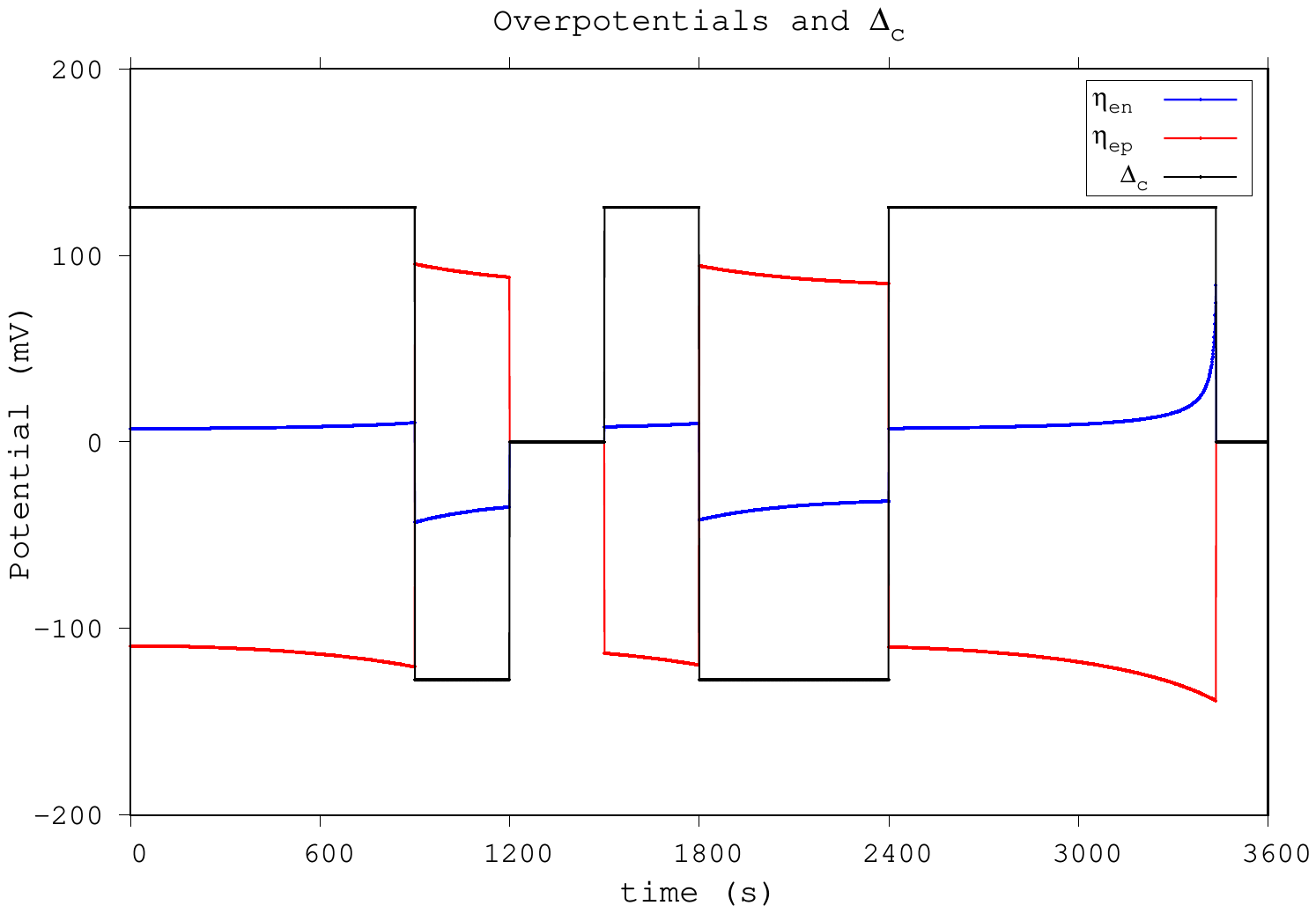}
    \caption{Over-potential and $\Delta_c$. Scenario THREE (left panel), scenario ONE (middle panel), and scenario HALF (right panel)}
    \label{fig:overpotential}
\end{figure}

\section{Conclusions}
We have mathematically developed a time-dependent model based on average variables and their values at the interface. Particular attention was paid to conservation and positivity, leading to the notion of diffusion-driven intensity that represents the major constraint on the range where the intensity is eligible. We presented a numerical method for the simulation and a representative example where the choice of the diffusion strongly modifies the capacity of the cell.

\section*{Funding}
S. Clain  was financially supported by the Fundação para a Ciência e a Tecnologia (Portuguese Foundation for Science and Technology) under the scope of the projects UID/00324/2025 \\ {\tt https://doi.org/10.54499/UID/00324/2025} (Centre for Mathematics of the University of Coimbra).\\[0.4em]

M. T. Malheiro was financially supported by the Fundação para a Ciência e a Tecnologia (Portuguese Foundation for Science and Technology) under the scope of the project UID/00013/2025 (https://doi.org/10.54499/UID/00013/2025). \\[0.4em]
C. M. Costa  was financially supported by the Fundação para a Ciência e a Tecnologia (Portuguese Foundation for Science and Technology) under the scope of the project 
UID/PRR/4650/2025 and UID/PRR2/04650/2025. \\[0.4em]
S. Clain and  M.T. Malheiro acknowledge the financial support of the Portuguese Foundation for Science and Technology (FCT) via a national funding for projects IC\&DT with the reference 2023.16854.ICDT.\\[0.4em]

\bibliographystyle{unsrt}
\bibliography{references.bib}
\appendix
\section{$b$-hyperbolic functions} \label{sec::b-trigonometry}
The Butler-Volmer relation involves the expression $\chi(\eta)=\exp(\beta f \eta)-\exp((\beta-1) f \eta)$ and  the inverse function is required in the algorithm to provide $\eta$ as a function of $r$. To this end, the present section concerns the development of special functions associated with the Butler-Volmer relation, we call the b-hyperbolic functions.
\begin{defn}
Let $\beta\in [0,1]$, we define the b-hyperbolic sine and cosine functions with
$$
\bsinh(x,\beta)=\frac{\exp(2\beta x)-\exp(2(\beta-1)x)}{2},\quad
\bcosh(x,\beta)=\frac{\exp(2\beta x)+\exp(2(\beta-1)x)}{2}.
$$
Such function are well-defined, and differentiable on $\mathbb R$ and one has $\chi(\eta)=2\,\bsinh(\eta\,f/2;\beta)$.
\end{defn}
The functions have several properties we list hereafter.
\begin{prop}
We have the algebraic relations
$$
\bsinh(x,\beta)=\exp((2\beta-1)x)\sinh(x),\quad \bcosh(x,\beta)=\exp((2\beta-1)x)\cosh(x),
$$
$$
\bsinh(x,1-\beta)=-\bsinh(-x,\beta),\quad \bcosh(x,1-\beta)=\bcosh(-x,\beta).
$$
Moreover, the derivatives read
\begin{eqnarray*}
\bsinh'(x,\beta)&=&(2\beta-1)\bsinh(x,\beta)+\bcosh(x,\beta)\\
\bcosh'(x,\beta)&=&(2\beta-1)\bcosh(x,\beta)+\bsinh(x,\beta)
\end{eqnarray*}
\end{prop}
In practice we need the inverse function of $\bsinh$ hence we have to determine the domain where the reciprocal function is well-defined and differentiable.
\begin{thm}
For any $\beta\in [0,1]$, the application $x\to y=\bsinh(x,\beta)$ is a bijection from $\mathbb R$ onto $\mathbb R$ and we denote by $y\to x=\arg \bsinh(y,\beta)$ the inverse function.
\end{thm}
{\sc Proof.} The derivative reads
\begin{eqnarray*}
\bsinh'(x,\beta)&=&(2\beta-1)\bsinh(x,\beta)+\bcosh(x,\beta)\\
                &=&\exp((2\beta-1) x)\Big [(2\beta-1)\sinh(x)+\cosh(x)\Big ]\\
                &=&\exp((2\beta-1) x)\Big [\beta \exp(x)+(1-\beta)\exp(-x)]
\end{eqnarray*}
Hence we have a convex combination of two positive functions and conclude that $\bsinh'(x,\beta)>0$ for all $x\in \mathbb R$.
Since $\bsinh(x,\beta)$ is a strictly increasing function from $\mathbb R$ onto $\mathbb R$, it admits an inverse function $y\to x=\absinh(y,\beta)$ from $\mathbb R$ onto $\mathbb R$, which is differentiable with strictly positive derivative.$\square$

\vskip 1em
{\bf Numerical computation of the inverse function}\\
Let $y\in\mathbb R$, we seek $x\in\mathbb R$ such that $y=\bsinh(x,\beta)$. Introducing $X=\exp(x)$, the equation reads
$$
2yX=X^{2\beta}\Big [X-\frac{1}{X} \Big ]\iff 2y=X^{\alpha+1}-X^{\alpha-1}
$$
with $\alpha=2\beta-1\in [-1,1]$. After algebraic manipulations, we seek $X(y)>0$ such that
$$
G(X(y);\alpha)=X^{\alpha+1}-2y-X^{\alpha-1}=0.
$$
Note that we recover the classical $\arg \sinh$ if $\beta=1/2$ $(\alpha=0$) since the expression turns to a quadratic polynomial. For $\alpha\neq 0$, the solution has to be determined numerically, and we use a simple Newton method. Let $X_p$ a sequence be given by $X_{p+1}=X_p-G(X_p;\alpha)/\partial_X G(X_p;\alpha)$ and read
\begin{eqnarray*}
X_{p+1}&=&X_p-\frac{X_p^{\alpha+1}-2y-X_p^{\alpha-1}}{(\alpha+1) X^{\alpha}-(\alpha-1)X^{\alpha-2}}.\\
&=&X_p\left [ 
1-\frac{X_p^{\alpha+1}-2y-X_p^{\alpha-1}}{(\alpha+1) X^{\alpha+1}-(\alpha-1)X^{\alpha-1}}
\right ].
\end{eqnarray*}
The approximations converge towards the unique solution $X(y)$ in a few iterations up to a prescribed tolerance while we take $X_0=1$ for the initialization.

\end{document}